\documentclass[amssymb, 11pt]{amsart}
\usepackage{amsmath}
\usepackage{amsthm}
\usepackage{amsfonts}
\usepackage{enumerate}

\usepackage{latexsym}

\newcommand{\Z}{\mathbb{Z}}
\newcommand{\FF}{\mathbb{F}}

\newlength{\standardunitlength}
\newtheorem{prop}{Proposition}[section]

\newtheorem{lemma}[prop]{Lemma}
\newtheorem{cor}[prop]{Corollary}
\newtheorem{theorem}[prop]{Theorem}

\begin{document}

\title{Derangements in Subspace Actions of Finite Classical Groups}

\author{Jason Fulman}
\address{Department of Mathematics\\
University of Southern California\\
Los Angeles, CA 90089-2532} \email{fulman@usc.edu}

\author{Robert Guralnick}
\address{Department of Mathematics\\
University of Southern California\\
Los Angeles, CA 90089-2532} \email{guralnic@math.usc.edu}

\thanks{Submitted: March 21, 2013, Revised: April 13, 2015}



\begin{abstract} This is the third in a series of four papers in which we prove a
 conjecture  made  by  Boston et al. and Shalev that the proportion of derangements (fixed point
 free elements) is bounded away from zero for transitive actions of
 finite simple groups on a set of size greater than one. This paper
 treats the case of primitive subspace actions. It is also shown that
 if the dimension and codimension of the subspace go to infinity, then
 the proportion of derangements goes to one. Similar results are
 proved for elements in finite classical groups in cosets of the
 simple group. The results in this paper have applications to
 probabilistic generation of finite simple groups and maps between
 varieties over finite fields. \end{abstract}

\maketitle

\section{Introduction}

    Let $G$ be a finite group and $X$ a transitive $G$-set.  An element
    $g \in G$ is called a derangement on $X$ if $g$ has no fixed
    points on $X$. We are interested in showing that under certain
    hypotheses the set of derangements of $G$ on $X$ is large. A
    conjecture due independently to Shalev and Boston et al. \cite{BDF}
    states that if $|X|>1$ and $G$ is simple, then there is a universal
    constant $\delta>0$ such that the proportion of elements of $G$ which
    are derangements on $X$ is at least $\delta$.

    We note that the study of derangements has applications to
    maps between varieties over finite fields and to random
    generation of groups. For details see the partially expository papers \cite{DFG}
    and \cite{FG}, as well as \cite{GW} and \cite{FG4}. Serre's survey \cite{Se}
    describes applications in number theory and topology. Finally, we remark that results in
    this paper are applied in the derangement
    problem for almost simple groups (where the analog of the Boston--Shalev
    conjecture fails); see for example \cite{FG3}.  See also the recent  paper
    \cite{KLS} for another application to  probabilistic generation of finite groups.

    In \cite{FG} we announced a proof of this conjecture and
    treated the case of finite Chevalley groups of bounded
    rank. The paper \cite{FG5} gives another proof in the bounded
    rank case and solves the Boston--Shalev conjecture for most other cases
    (all except subspace, extension field, and imprimitive group actions).
    The current paper treats the case of subspace actions, and the sequel
    \cite{FG2} treats extension field and imprimitive group actions,
    completing the proof of the Boston--Shalev conjecture. Prior to our
    work Shalev \cite{Sh} had analyzed many families of actions for the
    special case of $PGL(n,q)$. However the case of subspace actions was not
    considered, and we resolve it here. We also treat a
    generalization (useful for the applications to curves) in
    which we study the proportion of derangements in a coset $gH$
    of a simple group $H$ in a larger group $G$ with $G/H$ cyclic.

 Our main result give stronger results than the Boston--Shalev conjecture.

 \begin{theorem}  \label{thm:main}  Let $G$ be a simple classical group defined over $\FF_q$ with
  natural module $V$.  Let $\Gamma_k$ denote a $G$-orbit of nondegenerate
  or totally singular $k$-spaces with $k \le (1/2) \dim V$.
  \begin{enumerate}
  \item There exists an absolute constant $\delta > 0$ such that the proportion
  of elements in $G$ which are both regular semisimple and fix no element
   of $\Gamma_k$ is greater than $\delta$.
  \item  The limit of the proportion of derangements acting on $\Gamma_k$
  as $k \rightarrow \infty$ is $1$.
   \end{enumerate}
 \end{theorem}

Note that if $k$ is fixed, then in fact the proportion of derangements does not approach
$1$ (even for $q \rightarrow \infty$).   The first part of the previous theorem implies the Boston--Shalev
conjecture for subspace actions.

    Next we describe the contents of this paper. Section \ref{preliminaries}
discusses preliminaries which will be used freely throughout the paper.
These include cycle indices for the finite classical groups, enumerative
formulae for various types of irreducible polynomials, related generating
function identities, and generalities about asymptotics of generating
functions.

    Section \ref{alternating} studies random permutations. First it examines
the probability that a random permutation fixes (i.e. leaves invariant) a
$k$-set, extending a result of Dixon \cite{D} to cosets of the alternating
group. It also proves results about random permutations which will be
needed in showing asymptotic equidistribution of regular semisimple
elements of finite classical groups over cosets. Section \ref{otherWeyl}
gives analogs of results of Section \ref{alternating} for other Weyl
groups.

    Section \ref{maximaltori} describes relationships between maximal tori and conjugacy
classes in the Weyl group. A typical consequence is that the proportion of
elements of $GL(n,q)$ which are regular semisimple and fix (i.e. leave
invariant) a $k$-space is at most the proportion of permutations in $S_n$ which fix a
$k$-set.

Section \ref{Large Fields} handles the case that the field size goes to infinity
(either with fixed rank or increasing rank).   Here we can use some algebraic
geometry and algebraic group theory (although for the case of increasing
rank, one needs to get precise bounds).  In particular, we prove
Theorem \ref{thm: large q}  which includes Theorem \ref{thm:main} for
the case $q \rightarrow \infty$.

    Section \ref{regularsemisimple} focuses on the proportions of regular semisimple
elements in finite classical groups. It reviews and extends results of
\cite{GL} and \cite{FNP}. It derives a cycle index for
$\Omega^{\pm}(2n,q)$ in even characteristic. Section
\ref{regularsemisimple} also uses combinatorics of maximal tori and the
method of Section \ref{maximaltori} to prove equidistribution of regular
semisimple elements over cosets.

    Section \ref{nearlyregss} proves that with high probability, an element of a
finite classical group is {\it nearly regular semisimple}, that is regular
semisimple on some subspace of bounded codimension. Whereas Section
\ref{maximaltori} enables one to study proportions of elements of $G$
which are regular semisimple and derangements on $k$-spaces, the results of
Section \ref{nearlyregss} allow one to move beyond regular semisimple
elements (for fixed $q$ the proportion of regular semisimple elements
does  not tend to $1$). For example it will be shown that when $1 \leq k \leq
n/2$ with $k \rightarrow \infty$, the proportion of elements in $SL(n,q)$ which are
derangements on $k$-spaces tends to $1$. (Prior to this paper it was not even
known that this proportion was bounded away from zero). We expect our
results on nearly regular semisimple elements to have many other
applications.   We  prove that the proportion of derangements on
$k$-spaces is bounded away from $0$
without the results of this section (and so prove the original Boston--Shalev
conjecture for subspace actions).

    Section \ref{mainresults} applies the tools of earlier sections to
prove the Boston--Shalev conjecture (and a generalization for cosets) for
primitive subspace actions. It moreover shows that for such actions with
$|G|$ sufficiently large, the proportion of elements which are both semisimple regular
and derangements in a subspace
action  is at least $\delta =.016$ (and often much better). The paper
\cite{BDF} speculates that if one is only concerned with derangements (not necessarily
regular semisimple) then one may be able to take $\delta=2/7$ (for any transitive action
of a finite simple group).
Section \ref{mainresults} also shows
that as the dimension and codimension of the subspace grow, the proportion
of derangements tends to $1$. Given the earlier results for $q \rightarrow \infty$,
it suffices to deal with the case of $q$ fixed.

\section{Preliminaries} \label{preliminaries}

\subsection{Cycle indices} \label{prelimGLcycle}

To begin we review a tool which will be used throughout this paper: cycle
index generating functions of the finite classical groups. Modeled on the
cycle index of the symmetric groups (to be discussed in Section
\ref{alternating}), these generating functions were introduced for $GL$ in
\cite{K}, further exploited for $GL$ in \cite{St}, generalized to $U,Sp,O$
in \cite{Fu2}, and applied in \cite{Fu2},\cite{FNP},\cite{W2}. Section
\ref{regularsemisimple} will use a cycle index for $\Omega^{\pm}(2n,q)$ in
even characteristic. Work of Britnell (\cite{B1},\cite{B2}) gives cycle
indices for $SL$ and $SU$, and for some variants of orthogonal and
conformal groups (\cite{B3},\cite{B4}).

    The purpose of a cycle index is to study properties of a random group element
depending only on its conjugacy class; we illustrate the case of $GL(n,q)$
and refer the reader to the references in the previous paragraph for other
groups. As is explained in Chapter 6 of the textbook \cite{Her}, an
element $\alpha \in GL(n,q)$ has its conjugacy class determined by its
rational canonical form. This form corresponds to the following
combinatorial data. To each monic, non-constant, irreducible polynomial
$\phi$ over the finite field $\mathbb{F}_q$, associate a partition
(perhaps the trivial partition) $\lambda_{\phi}$ of some non-negative
integer $|\lambda_{\phi}|$. Let $\deg(\phi)$ denote the degree of $\phi$.
The only restrictions necessary for this data to represent a conjugacy
class are that $|\lambda_z| = 0$ and $\sum_{\phi} |\lambda_{\phi}|
\deg(\phi) = n.$ Note that given a matrix $\alpha$, the vector space $V$
on which it acts uniquely decomposes as a direct sum of spaces $V_{\phi}$
where the characteristic polynomial of $\alpha$ on $V_{\phi}$ is a power
of $\phi$ and the characteristic polynomials on different summands are
coprime. Each $V_{\phi}$ decomposes as a direct sum of cyclic subspaces,
and the parts of $\lambda_{\phi}$ are the dimensions of the subspaces in
this decomposition divided by the degree of $\phi$. For example, the
identity matrix has $\lambda_{z-1}$ equal to $(1^n)$ and all other
$\lambda_{\phi}$ equal to the empty set. An elementary transvection with
$a \neq 0$ in the $(1,2)$ position, ones on the diagonal, and zeros
elsewhere has $\lambda_{z-1}$ equal to $(2,1^{n-2})$ and all other
$\lambda_{\phi}$ equal to the empty set. For a given matrix only finitely
many $\lambda_{\phi}$ are non-empty.

Many algebraic properties of a matrix can be stated in terms of the data
parameterizing its conjugacy class. For instance the characteristic
polynomial of $\alpha \in GL(n,q)$ is equal to $\prod_{\phi}
\phi^{|\lambda_{\phi}(\alpha)|}$ and the minimal polynomial of $\alpha$ is
equal to $\prod_{\phi} \phi^{\lambda_{\phi,1}(\alpha)}$ where
$\lambda_{\phi,1}$ is the largest part of the partition $\lambda_{\phi}$.
Furthermore $\alpha \in GL(n,q)$ is semisimple if and only if all
$\lambda_{\phi}(\alpha)$ have largest part at most 1, regular if and only
if all $\lambda_{\phi}(\alpha)$ have at most 1 part, and regular
semisimple if and only if all $\lambda_{\phi}(\alpha)$ have size at most
1.
        To define the cycle index for $Z_{GL(n,q)}$, let
$x_{\phi,\lambda}$ be variables corresponding to pairs of polynomials and
partitions. Define \[ Z_{GL(n,q)} = \frac{1}{|GL(n,q)|} \sum_{\alpha \in
GL(n,q)} \prod_{\phi} x_{\phi,\lambda_{\phi}(\alpha)}. \] Note that the
coefficient of a monomial is the probability of belonging to the
corresponding conjugacy class, and is therefore equal to one over the
order of the centralizer of a representative. Let $m_i(\lambda)$ be the
number of parts of size $i$ of $\lambda$, and let $\lambda_i'$ denote the
number of parts of $\lambda$ of size at least $i$. It is well known (e.g.
easily deduced from page 181 of \cite{Mac}) that the size of the conjugacy
class of $GL(n,q)$ corresponding to the data $\{\lambda_{\phi}\}$ is \[
\frac{|GL(n,q)|}{\prod_{\phi} q^{\deg(\phi) \cdot \sum_i
(\lambda_{\phi,i}')^2} \prod_{i \geq 1}
(\frac{1}{q^{\deg(\phi)}})_{m_i(\lambda_{\phi})}},\] where
\[ (1/q)_j = (1-1/q)(1-1/q^2) \cdots (1-1/q^j).\] It follows that
\[ 1+\sum_{n=1}^{\infty} Z_{GL(n,q)} u^n
= \prod_{\phi \neq z} \left[\sum_{\lambda} x_{\phi,\lambda}
\frac{u^{|\lambda| \cdot \deg(\phi)}}{q^{\deg(\phi) \cdot \sum_i
(\lambda_i')^2} \prod_{i \geq 1}
\left(\frac{1}{q^{\deg(\phi)}}\right)_{m_i(\lambda)}}\right]. \]
This is called the cycle index generating function.

\subsection{Polynomial enumeration and related identities} \label{polyenum}

    Next we recall some results about enumeration of various types of irreducible
polynomials and related identities. (The enumerative results are only
really required in Section \ref{nearlyregss} and only upper bounds are
used. However as exact formulas are available, we state them). Let
$N(q;d)$ denote the number of monic irreducible degree $d$ polynomials
over $\mathbb{F}_q$ with non-zero constant term. Let $\mu$ denote the
Moebius function of elementary number theory. The following result is well
known and appears for example in \cite{LN}.

\begin{lemma} $N(q;1)=q-1$ and for $d>1$, $N(q;d)=\frac{1}{d} \sum_{r|d} \mu(r)q^{d/r}$.
\end{lemma}

    Next we consider analogous results useful for treating the unitary groups.
Let $\sigma: x \mapsto x^q$ be the involutory automorphism of
$\mathbb{F}_{q^2}$. This induces an automorphism of the polynomial ring
$\mathbb{F}_{q^2}[z]$ by sending $\sum_{0 \leq i \leq n} a_i z^i$ to
$\sum_{0 \leq i \leq n} a_i^{\sigma} z^i$. Then given a polynomial
$\phi(z)$ with coefficients in the field $\mathbb{F}_{q^2}$ and non-zero
constant term, define an involutory map $\phi \mapsto \tilde{\phi}$ by \[
\tilde{\phi}(z) = \frac{z^{\deg(\phi)}
\phi^{\sigma}(1/z)}{\phi(0)^{\sigma}}.\] The polynomial $\tilde{\phi}$ is
called the conjugate of $\phi$.

    Let $\tilde{N}(q;d)$ denote the number of monic irreducible self-conjugate
degree $d$ polynomials over $\mathbb{F}_{q^2}$. Let $\tilde{M}(q;d)$
denote the number of (unordered) conjugate pairs $\{\phi,\tilde{\phi}\}$
of monic irreducible polynomials of degree $d$ over $\mathbb{F}_{q^2}$
that are not self-conjugate.

\begin{lemma} \rm({\cite[Theorem 9]{Fu2})}
\begin{enumerate}
\item \[ \tilde{N}(q;d) = \left\{ \begin{array}{ll}

        0 & \mbox{if
$d$ \ is \ even}\\

                \frac{1}{d} \sum_{r|d} \mu(r) (q^{d/r}+1)
                              & \mbox{if $d$ \ is \ odd }

                                                \end{array}
                        \right.                  \]

\item \[ \tilde{M}(q;d) = \left\{ \begin{array}{ll}

        \frac{1}{2}(q^2-q-2) & \mbox{if
$d=1$}\\

\frac{1}{2d} \sum_{r|d} \mu(r) (q^{2d/r}-q^{d/r}) & \mbox{if $d$ \ is \
odd \ and \ $d>1
$}\\

                \frac{1}{2d} \sum_{r|d} \mu(r) q^{2d/r}  & \mbox{if $d$ \ is \ even}

                                                \end{array}
                        \right.                  \]

\end{enumerate}
\end{lemma}

    Finally we consider analogous results useful for treating the symplectic and
orthogonal groups. Given a degree $n$ monic polynomial $\phi(z)$ with
$\phi(0) \neq 0$, define its conjugate $\phi^*(z) := \frac{z^n
\phi(1/z)}{\phi(0)}$.  Let $N^*(q;d)$ denote the number of monic
irreducible self-conjugate polynomials of degree $d$ over $\mathbb{F}_q$,
and let $M^*(q;d)$ denote the number of (unordered) conjugate pairs
$\{\phi,\phi^*\}$ of monic, irreducible, non-self-conjugate polynomials of
degree $d$ over $\mathbb{F}_q$.

\begin{lemma} \rm{(\cite{FNP})} Let $f= \gcd(q-1,2)$.
\begin{enumerate}
\item \[ N^*(q;d) = \left\{ \begin{array}{ll}

        f & \mbox{if
$d=1$}\\

                                        0 & \mbox{if $d$ \ is \ odd \ and \ $d>1
$}\\

                \frac{1}{d} \sum_{r|d \atop r \ odd} \mu(r) (q^{d/(2r)}+1-
f)              & \mbox{if $d$ \ is \ even }

                                                \end{array}
                        \right.                  \]

\item \[ M^*(q;d) = \left\{ \begin{array}{ll}

        \frac{1}{2}(q-f-1) & \mbox{if
$d=1$}\\

              \frac{1}{2} N(q;d) & \mbox{if $d$ \ is \ odd \ and \ $d>1
$}\\

                \frac{1}{2} (N(q;d)-N^*(q;d))  & \mbox{if $d$ \ is \ even}

                                                \end{array}
                        \right.                  \]
\end{enumerate}
\end{lemma}

    The following generating function identities will be useful.
Lemma \ref{polynomialidentity} is well known; see for instance \cite{Fu2}.

\begin{lemma} \label{polynomialidentity} Suppose that $|u|<q^{-1}$. Then
\[ \prod_{d \geq 1} (1-u^d)^{-N(q;d)} = \frac{1-u}{1-uq}.\]
\end{lemma}

\begin{lemma} \label{set1incycle} Suppose that $|u|<q^{-1}$.
\begin{enumerate}
\item $\prod_{d \geq 1} \prod_{i
\geq 1} \left(1-\frac{u^d}{q^{id}} \right)^{-N(q;d)} = (1-u)^{-1}$

\item $\prod_{d \geq 1 } \prod_{i
\geq 1} \left(1+\frac{u^d (-1)^i}{q^{id}}\right)^{-\tilde{N}(q;d)} \left(
1-\frac{u^{2d}}{q^{2id}} \right)^{-\tilde{M}(q;d)}= (1-u)^{-1}$

\item Let $f=\gcd(q-1,2)$.   Then
\begin{eqnarray*} & & \prod_{i \geq 1}
\left(1-\frac{u}{q^{2i-1}} \right)^{-f} \prod_{d \geq 1} \prod_{i \geq 1}
\left(1+\frac{(-1)^i u^d}{q^{id}}\right)^{-N^*(q;2d)}
\left(1-\frac{u^{d}}{q^{id}}\right)^{-M^*(q;d)} \\ & = &
(1-u)^{-1}\end{eqnarray*}

\end{enumerate}
\end{lemma}

\begin{proof} For the first assertion, note by Lemma \ref{polynomialidentity} that
\begin{eqnarray*}
 \prod_{d \geq 1} \prod_{i
\geq 1} \left(1-\frac{u^d}{q^{id}} \right)^{-N(q;d)}  & = & \prod_{i \geq 1} \prod_{d
\geq 1} \left(1-\frac{u^d}{q^{id}} \right)^{-N(q;d)} \\
& = & \prod_{i \geq 1} \frac{(1-u/q^i)}{(1-u/q^{i-1})} \\
& = & (1-u)^{-1}.
\end{eqnarray*}

For the second assertion, the left hand side is equal to
\begin{eqnarray*}
& & \prod_{i \ odd} \prod_{d \geq 1} \left(1 - \frac{u^d}{q^{id}}\right)^{-\tilde{N}(q;d)} \left(
1-\frac{u^{2d}}{q^{2id}} \right)^{-\tilde{M}(q;d)} \\
& & \cdot \prod_{i \ even} \prod_{d \geq 1} \left(1 + \frac{u^d}{q^{id}}\right)^{-\tilde{N}(q;d)} \left(
1-\frac{u^{2d}}{q^{2id}} \right)^{-\tilde{M}(q;d)}.
\end{eqnarray*} By parts (a) and (c) of Lemma 1.3.14 of \cite{FNP}, this is equal to
\[ \prod_{i \ odd} \frac{(1+u/q^i)}{(1-qu/q^i)} \prod_{i \ even} \frac{(1-u/q^i)}{(1+qu/q^i)} = (1-u)^{-1} .\]

For the third assertion,
\[ \prod_{d \geq 1} \prod_{i \geq 1}
\left(1+\frac{(-1)^i u^d}{q^{id}}\right)^{-N^*(q;2d)}
\left(1-\frac{u^{d}}{q^{id}}\right)^{-M^*(q;d)} \] is equal to
\begin{eqnarray*}
& & \prod_{i \ odd} \prod_{d \geq 1} \left(1 - \frac{u^d}{q^{id}}\right)^{-N^*(q;2d)}
\left(1-\frac{u^{d}}{q^{id}}\right)^{-M^*(q;d)} \\
& & \cdot \prod_{i \ even} \prod_{d \geq 1} \left(1 + \frac{u^d}{q^{id}}\right)^{-N^*(q;2d)}
\left(1-\frac{u^{d}}{q^{id}}\right)^{-M^*(q;d)}.
\end{eqnarray*} By parts (a) and (d) of Lemma 1.3.17 of \cite{FNP}, this is equal to
\[ \prod_{i \ odd} \frac{(1-u/q^i)^f}{(1-qu/q^i)} \prod_{i \ even} (1-u/q^i) =  \frac{\prod_{i \ odd} (1-u/q^i)^f}{1-u}.\]
\end{proof}

    The statement of Lemma \ref{Stongs} uses the partition notation of Subsection \ref{prelimGLcycle}.

\begin{lemma} \rm{(\cite{St})} \label{Stongs} \[ 1+\sum_{\lambda} \frac{u^{|\lambda|}}
{q^{\sum_i (\lambda_i')^2} \prod_i (1/q)_{m_i(\lambda)}} =
\prod_{i \geq 1} \frac{1}{1-u/q^i}.\] \end{lemma}

    We also record the following identity as it will be needed.

\begin{lemma} \label{Euleridentity} \rm{(Euler)}
\begin{enumerate}
\item \[ \prod_{i \geq 1} (1-\frac{u}{q^i}) = \sum_{n=0}^{\infty}
\frac{(-u)^n}{(q^n-1)\cdots(q-1)}.\]
 \item \[ \prod_{i \geq 1}
(1-\frac{u}{q^i})^{-1} = \sum_{n=0}^{\infty} \frac{u^n q^{{n \choose
2}}}{(q^n-1)\cdots (q-1)}.\]
\end{enumerate} \end{lemma}

The following lemma is Euler's pentagonal number theorem (see for instance
page 11 of \cite{A1}).

\begin{lemma} \label{pent} For $q>1$,
\begin{eqnarray*} \prod_{i \geq 1} (1-\frac{1}{q^i}) & = & 1 +
\sum_{n=1}^{\infty} (-1)^n (q^{-\frac{n(3n-1)}{2}}+
q^{-\frac{n(3n+1)}{2}})\\ & = & 1-
q^{-1}-q^{-2}+q^{-5}+q^{-7}-q^{-12}-q^{-15}+ \cdots \end{eqnarray*}
\end{lemma}

Throughout this paper quantities which can be easily re-expressed in
terms of the infinite product $\prod_{i=1}^{\infty} (1-\frac{1}{q^i})$
will sometimes arise, and Lemma \ref{pent} gives arbitrarily accurate upper
and lower bounds on these products. Hence we will state bounds like
$$
\prod_{i=1}^{\infty} (1+\frac{1}{2^i}) = \prod_{i=1}^{\infty}
\frac{(1-\frac{1}{4^i})}{(1-\frac{1}{2^i})} \leq 2.4
$$
 without explicitly
mentioning Euler's pentagonal number theorem on each occasion.

\subsection{Generating function asymptotics} \label{asymptot}

    If one is given two generating functions $f(u)=\sum_{n \geq 0} f_nu^n$ and
$g(u)=\sum_{n \geq 0} g_n u^n$, the notation $f << g$ will mean that $|f_n|
\leq |g_n|$ for all $n$. This will be used throughout this paper.

    In determining the limiting probabilities of the generating functions
considered in this paper, we shall sometimes use the following standard
result \ about analytic functions.

\begin{lemma} \label{Taylorcoeff}
Suppose that $g(u) = \sum_{n=0}^\infty\, a_n u^n$ and $g(u) = f(u)/(1 -
u)$ for\
 $|u| < 1$. Let $D(R)$ denote the open disc
consisting of complex numbers $u$ with $|u|<R$.
If $f(u)$ is analytic in $D(R)\kern 1pt,$ where $R > 1\kern 1pt,$
then $\lim_{n \rightarrow \infty} a_n = f(1)$
and $|a_n - f(1)| = o(r^{-n})$ for any $r$ such that $1 < r < R$.
\end{lemma}

\begin{proof}
Define $F(1) = f'(1)$ and $F(u) = (f(1) - f(u))/(1 - u)$ elsewhere in
$D(R)$. Then $F$ is analytic in that disc and must be represented by a
Taylor series $\sum_n b_n u^n$ converging there. If $1 < r < R$ then $\sum
b_n r^n$ converges and so $b_n r^n \to 0$ as $n \to\
 \infty$, that is $|b_n| = o(r^{-n})$.
Now $g(u) = f(1)/(1 - u) - F(u)$ and therefore $a_n = f(1) - b_n$.
Thus $a_n \to f(1)$ and $|a_n - f(1)| = o(r^{-n})$ as $n \to \infty$.
\end{proof}

\section{Alternating and symmetric groups} \label{alternating}

    This section studies conjugacy class properties of random permutations.
It begins by reviewing the cycle index of the symmetric groups. Then it
discusses the probability that a random permutation fixes a $k$-set,
extending an upper bound of Dixon to cosets of the alternating group. It
also discusses an upper bound of Luczak and Pyber. Finally, this section proves results
about random permutations which will be needed in showing asymptotic
equidistribution of regular semisimple elements of finite classical groups
over cosets.

\subsection{Cycle index of the symmetric groups}

    For a permutation $\pi$ let $n_i(\pi)$ be the number of length $i$ cycles of $\pi$.
Polya proved that \[ 1+\sum_{n=1}^{\infty} \frac{u^n}{n!} \sum_{\pi \in
S_n} \prod_{i \geq 1} x_i^{n_i(\pi)} = \prod_{m \geq 1}
e^{x_mu^m/m}.\]
This follows from the fact that the number of
permutations in $S_n$ with $n_i$ cycles of size $i$ is equal to
$$\frac{n!}{\prod_{i} i^{n_i} n_i!}.$$
 This generating function is called
the cycle index of the symmetric groups, because it stores complete
information about the cycle structure of permutations. This cycle index
will be used several times in this paper.

An integer valued random variable $Z$ is said to be Poisson of mean
$\lambda$ if the chance that $Z=k$ is $\frac{\lambda^k}{e^{\lambda} k!}$.
The following result of Shepp and Lloyd will be important. For in-depth
discussions of Theorem \ref{Poissonlimit}, see \cite{DP} or \cite{AT}.
Theorem \ref{Poissonlimit} follows from the cycle index of the symmetric
groups and Lemma \ref{Taylorcoeff}.

\begin{theorem} \label{Poissonlimit} \rm{(\cite{ShLl})}
 Given a permutation $\pi$,
let $n_i(\pi)$ denote the number of $i$-cycles of $\pi$. Then for fixed
$k$ and $\pi$ random in $S_n$, the vector $(n_1(\pi),\cdots,n_k(\pi))$
converges as $n \rightarrow \infty$ to $(Z_1,\cdots,Z_k)$, where the $Z_i$
are independent and $Z_i$ is Poisson with mean $1/i$. \end{theorem}

\subsection{Chance of fixing a $k$-set}

    Motivated by questions about random generation and computation of Galois groups,
Dixon \cite{D} examined the probability that a random permutation fixes
(i.e. leaves invariant) a $k$-set. Note that we can suppose that $1 \leq k
\leq n/2$, since a permutation fixes a $k$-set if and only if it fixes an
$(n-k)$-set.

\begin{theorem} \label{Dix} \rm{(\cite{D})} For $1 \leq k \leq n/2$, the proportion of
elements in $S_n$ which fix a $k$-set is at most $2/3$. \end{theorem}

    Our next goal is to prove that for $n \geq 5$ (so that $A_n$ is simple), the
proportion of derangements in a coset $gA_n$ of $A_n$ in $S_n$ on $k$-sets is at least
$1/3$.

\begin{lemma} \label{firstcycle} Let $gA_n$ be a coset of $A_n$ in $S_n$.
Then for $1 \leq j \leq n-2$, the proportion of elements in $gA_n$ with
the property that the cycle containing 1 has length $j$ is $\frac{1}{n}$.
\end{lemma}

\begin{proof} There are ${n-1 \choose j-1}$ ways of choosing the elements to be in the
cycle with 1 and $(j-1)!$ ways of ordering them. Since $n-j \geq 2$, the
number of elements in either coset of $A_{n-j}$ in $S_{n-j}$ is
$(n-j)!/2$. The result now follows since \[ \frac{{n-1 \choose j-1} (j-1)!
(n-j)!/2}{|A_n|}=1/n.\] \end{proof}

\begin{theorem} \label{altkgeq2} Let $gA_n$ be a coset of $A_n$ in $S_n$.
\begin{enumerate}
\item For $2 \leq k \leq n/2$, the proportion of elements in $gA_n$ which are
derangements on $k$-sets is at least $1/3$.
\item For $n \geq 5$, $k = 1$, the proportion of elements in $gA_n$ without fixed
points is at least $1/3$.
\end{enumerate} In particular, when $A_n$ is simple, the proportion of derangements
in a coset $gA_n$ on $k$-sets ($1 \leq k \leq \frac{n}{2}$) is at least
$1/3$.
\end{theorem}

\begin{proof} For the proof of part 1 we use a method similar to that of Dixon \cite{D}.
Let $I(n,k)$ be the set of elements in $gA_n$ which leave invariant a
$k$-set and let $i(n,k)=\frac{|I(n,k)|}{|gA_n|}$ be the proportion of such
elements. Let $C(n,j)$ be the set of permutations in $gA_n$ such that the
cycle containing 1 has size j. Consider the set $I(n,k) \cap C(n,j)$. For
$n-k<j$ this set is empty. For $1 \leq j \leq k$ or $j=n-k-1,n-k$ note
that $|I(n,k) \cap C(n,j)| \leq |C(n,j)| = \frac{|gA_n|}{n}$ by Lemma
\ref{firstcycle} since $j \leq n-2$. For each $j$ satisfying $k+1 \leq j
\leq n-k-2$, observe that a fixed $k$-set must use only symbols outside of
the cycle containing 1. Thus the proportion of such elements is at most
$\frac{1}{n} i(n-j,k)$. Now use induction on $n$. The base case $n=4$ is
easily checked, and $i(n-j,k) \leq 2/3$ if $k \leq (n-j)/2$ and otherwise
$i(n-j,k)=i(n-j,n-j-k) \leq 2/3$ since $n-j-k \geq 2$. Thus
\[ i(n,k) \leq \frac{k+2+(n-2k-2)2/3}{n} \leq 2/3
\] since $k \geq 2$.

    For the proof of part 2 we use the cycle index of the alternating groups.
This is the average of the cycle index of the symmetric groups and the
cycle index of the symmetric groups with $x_i$ replaced by $-x_i$ for $i$
even. Setting $x_1=0$ and $x_i=1$ for $i \geq 2$ gives that the proportion
of derangements (on $1$-sets) in $A_n$ is the coefficient of $u^n$ in
\[ \prod_{i \geq 2} e^{u^i/i} + \prod_{i \geq 2} e^{(-1)^{i+1} u^i/i} =
\frac{1}{e^u(1-u)} + \frac{1+u}{e^u}.\] Using the power series expansion
for $e^{-u}$, it is straightforward to see that for $n \geq 5$, this
coefficient is at least $1/3$. Similarly, for the other coset of $A_n$ in $S_n$,
the proportion of derangements is the coefficient of $u^n$ in
\[ \frac{1}{e^u(1-u)} - \frac{1+u}{e^u},\] which is at least $1/3$ for
$n \geq 5$. \end{proof}

    Concerning large $k$, we will need the following result of Luczak and Pyber
\cite{LP}, which shows that as $k \rightarrow \infty$, the proportion of
elements in $S_n$ which are derangements on $k$-sets approaches 1.

\begin{theorem} \label{LPyber} \rm{(\cite{LP})} There is a universal constant $A$ such
that the probability that a random element of $S_n$ fixes a $k$-set is at
most $Ak^{-.01}$ for $1 \leq k \leq n/2$. \end{theorem}

To close this subsection, we establish results which will be useful in
analyzing subspace actions of $SL(n,3)$.

\begin{lemma} \label{usingcycle} For $n \geq 2$, the chance that an element of $S_n$
has 1 or 2 fixed points is at most $3/5$.
\end{lemma}

\begin{proof} For $n=2,3,4$ one checks this directly. For $n \geq 5$, it follows from the
cycle index (or from inclusion-exclusion) that the proportion of elements
with 1 fixed point is $\sum_{i=0}^{n-1} (-1)^i/i! \leq 1/2-1/6+1/24$ and
 that the  proportion of elements with 2 fixed points is $\frac{1}{2}
\sum_{i=0}^{n-2} (-1)^i/i! \leq .5(1/2-1/6+1/24)$. Adding these bounds
together gives $.5625 \leq 3/5$. \end{proof}

\begin{lemma} \label{atmost2fixed} For $2 \leq k \leq n/2$,
the chance that an element in $S_n$ fixes a $k$-set and has at most 2 fixed
points is at most $3/5$. \end{lemma}

\begin{proof} The method of proof is an induction along the lines of Lemma 2 of Dixon
\cite{D}. Let $I(n,k)$ be the set of elements in $S_n$ which fix a $k$-set
and have at most 2 fixed points, and $i(n,k) = \frac{|I(n,k)|}{n!}$. Let
$C(n,j)$ be the set of permutations such that the cycle containing the
element 1 has size $j$. Now consider the set $I(n,k) \cap C(n,j)$. This
set is empty for $n-k<j$. For $k+1 \leq j \leq n-k-2$, any element of this
set fixes a $k$-set which is disjoint from the cycle containing 1. Thus
$|I(n,k) \cap C(n,j)| = (n-1)! i(n-j,k)$, and using the fact that
$i(n-j,k)=i(n-j,n-j-k)$ it follows by induction that in this case $|I(n,k)
\cap C(n,j)| \leq \frac{3}{5} (n-1)!$. If $j=n-k-1$ then $|I(n,k) \cap
C(n,j)| \leq \frac{3}{5} (n-1)!$, since by Lemma \ref{usingcycle},
the proportion of $\pi \in S_{n-j}$ with one or two fixed points is $\leq 3/5$.
Finally, if $j \leq k$ or $j=n-k$, then $|I(n,k) \cap C(n,j)| \leq
|C(n,j)| = (n-1)!$. Hence
\[ i(n,k) \leq (k+1+3(n-2k-1)/5)/n \leq 3/5 \] where the second
inequality follows because $k \geq 2$. \end{proof}

\subsection{Other results on random permutations}

    This subsection derives a result on random permutations which will be useful in
analyzing how the proportion of regular semisimple elements varies over
cosets of $SL(n,q)$ in $GL(n,q)$.

    We begin with two lemmas which bound coefficients of certain generating functions.

\begin{lemma} \label{binomialbound} For $0 < t <1, r \geq 1$, the coefficient of $u^r$
in $(1-u)^{-t}$ is at most $\frac{t}{r} e^{t} (r)^{t}$.
\end{lemma}

\begin{proof} This coefficient is equal to $\frac{t}{r} \prod_{i=1}^{r-1}
(1+\frac{t}{i})$. Taking logarithms (base $e$), one sees that
\begin{eqnarray*}
\log \left[ \prod_{i=1}^{r-1} \left( 1+\frac{t}{i} \right) \right] & = &
 \sum_{i=1}^{r-1} \log \left( 1+\frac{t}{i} \right)\\
& \leq & \sum_{i=1}^{r-1} \frac{t}{i} \leq t (1+ \log(r-1)).
\end{eqnarray*} Taking exponentials one sees that the sought proportion is at
most $\frac{t}{r} e^{t} (r)^{t}$.
\end{proof}

Recall the notation $<<$ defined in Subsection \ref{asymptot}.

\begin{lemma} \label{likeKnuth} For $p \geq 2$ fixed, the coefficient
of $u^n$ in $\exp \left(\sum_{i \geq 1} \frac{u^i}{pi^2} \right)$ is $O
\left( \frac{\log(n)}{pn} \right)$.\end{lemma}

\begin{proof} Let $f(u)=\exp \left(\sum_{i \geq 1} \frac{u^i}{pi^2}
\right)$, and let $f_n$ denote the coefficient of $u^n$ in $f(u)$.
Considering the coefficient of $u^{n-1}$ in the derivative of $f(u)$, one
obtains the recursion \[ pnf_n = \sum_{j=0}^{n-1} f_j
\frac{1}{n-j}.\] Since \[ f(u) << exp \left( \sum_{i \geq 1} u^i/i \right)= \frac{1}{1-u},\]
it follows that $f_n \leq 1$. This with the recursion gives that
$f_n=O \left( \frac{\log(n)}{pn} \right)$, as claimed.
\end{proof}

\begin{theorem} \label{sntorierror} Let $a_1,\cdots,a_r$ be the distinct cycle
lengths of a permutation and let $m_1,\cdots,m_r$ be the multiplicities
with which they occur. Then the proportion of $\pi \in S_n$ satisfying
$\gcd(a_1m_1,\cdots,a_rm_r,q-1) \neq 1$ is at most $\frac{c_1
\log(n)^3}{n^{1/2}}$ for a universal constant $c_1$ (independent of
$n,q$).
\end{theorem}

\begin{proof} Letting $p$ be a prime, we show that the proportion of $\pi
\in S_n$ with $\gcd(a_1m_1,\cdots,a_rm_r)$ divisible by $p$ is at most
$\frac{c_1 \log(n)^2}{n^{1/2}}$ for a universal constant $c_1$. This is
enough since $n$ has at most $\log_2(n)$ distinct prime factors.

The proportion of permutations satisfying $\gcd(a_1m_1,\cdots,a_rm_r)$
divisible by $p$ is at most the proportion of permutations where all
cycles of length not divisible by $p$ occur with multiplicity a multiple
of $p$. From the cycle index of the symmetric groups, the latter
proportion is the coefficient of $u^n$ in
\begin{eqnarray*} & & \prod_{i \geq 1} e^{\frac{u^{ip}}{ip}} \prod_{i
\geq 1 \atop gcd(i,p)=1} \left( 1+\frac{u^{ip}}{i^p
p!}+\frac{u^{2ip}}{i^{2p}(2p)!}+ \cdots \right)\\ & = & (1-u^p)^{-1/p}
\prod_{i \geq 1 \atop gcd(i,p)=1} \left( 1+\frac{u^{ip}}{i^p
p!}+\frac{u^{2ip}}{i^{2p}(2p)!}+ \cdots \right)\\ & << & (1-u^p)^{-1/p}
\prod_{i \geq 1} \left( 1+\frac{u^{ip}}{i^p
p}+\frac{u^{2ip}}{i^{2p}p^22!}+ \cdots \right)\\
& = & (1-u^p)^{-1/p} \exp \left( \sum_{i \geq 1} \frac{u^{ip}}{pi^p} \right) \\
& << & (1-u^p)^{-1/p} \exp \left( \sum_{i \geq 1} \frac{u^{ip}}{pi^2}
\right).
\end{eqnarray*} This is simply the coefficient of $u^{n/p}$ in
\[ (1-u)^{-1/p} \exp \left( \sum_{i \geq 1} \frac{u^i}{pi^2} \right).\]
It follows from Lemmas \ref{binomialbound} and \ref{likeKnuth} that the
sought coefficient is at most
\[ C \left[ \frac{\log(n)}{n} + n^{-1/2} + \sum_{r=1}^{(n/p)-1}
r^{-1/2} \cdot \frac{\log(n)}{n-pr} \right] \] for a universal constant
$C$. Note that the first term came from an upper bound for the coefficient of
$u^{n/p}$ in $\exp \left( \sum_{i \geq 1} \frac{u^i}{pi^2} \right)$, and
that the second term came from  an upper bound for the coefficient of $u^{n/p}$
in $(1-u)^{-1/p}$. Splitting the sum into two sums (one with $r$ ranging
from $1$ to $\frac{n}{2p}$ and the other with $r$ ranging from
$\frac{n}{2p}+1$ to $\frac{n}{p}-1$) proves that the proportion of
permutations with $\gcd(a_1m_1,\cdots,a_rm_r)$ divisible by $p$ is at most
$O \left( \frac{\log(n)^2}{n^{1/2}} \right)$, as claimed.
\end{proof}

\section{Results for other Weyl groups} \label{otherWeyl}

    This section extends results of Section \ref{alternating} to other
Weyl groups, and considers various analogs of the property that a random
permutation fixes a $k$-set.

    To begin we review the cycle index of the hyperoctahedral group $B_n$.
Given an element $\pi \in B_n$, let $n_i(\pi)$ be the number of positive
$i$-cycles of $\pi$ and let $m_i(\pi)$ be the number of negative
$i$-cycles of $\pi$. From \cite{JK}, the conjugacy classes of $B_n$ are
indexed by pairs of $n$-tuples $(n_1,\cdots,n_n)$ and $(m_1,\cdots,m_n)$
satisfying $\sum_i i(n_i+m_i)=n$, and a conjugacy class with this data has
size
\[ \frac{2^n n!}{\prod_i n_i! m_i! (2i)^{n_i+m_i}}.\] As noted in
\cite{DP}, this can be conveniently encoded by the equation
\[ 1+\sum_{n \geq 1} \frac{u^n}{2^n n!} \sum_{\pi \in B_n} \prod_{i \geq 1}
x_i^{n_i(\pi)} y_i^{m_i(\pi)} = \prod_{i \geq 1} e^{\frac{u^i
(x_i+y_i)}{2i}}.\] This equation is referred to as the cycle index of the
hyperoctahedral groups.

In analogy with Theorem \ref{Poissonlimit}, Diaconis and Pitman obtained
the following result.

\begin{theorem} \label{hyperoctPoisson} Given $\pi \in B_n$, let $n_i(\pi),m_i(\pi)$
denote the number of positive and negative i-cycles of $\pi$ respectively.
Then for fixed $k$ and $\pi$ random in $B_n$, the vector
$(n_1(\pi),m_1(\pi),\cdots,n_k(\pi),m_k(\pi))$ converges as $n \rightarrow
\infty$ to $(Y_1,Z_1,\cdots,Y_k,Z_k)$ where all the $Y's,Z's$ are
independent and $Y_i,Z_i$ are both Poisson random variables with mean $1/(2i)$. \end{theorem}

Lemma \ref{alleven} and Theorem \ref{evenkset} will be useful in analyzing
the action of the unitary groups on totally singular $k$-spaces.

\begin{lemma} \label{alleven}
\begin{enumerate}
\item The proportion of elements in $S_{2k}$ with all cycles even is
${2k \choose k}/4^k$.

\item The proportion of part 1 is decreasing in $k$ so is maximized for $k=1$
when it is equal to $1/2$.

\item The proportion of part 1 is at most $\frac{1}{(\pi k)^{1/2}}
e^{\frac{1}{24k}-\frac{2}{12k+1}} < \frac{1}{(\pi k)^{1/2}}$ and is
asymptotic to $\frac{1}{(\pi k)^{1/2}}$.
\end{enumerate}
\end{lemma}

\begin{proof} From the cycle index of the symmetric groups (reviewed
in Section \ref{alternating}), it follows that the sought proportion is the
coefficient of $u^{2k}$ in
$$
\prod_{i \geq 1} e^{\frac{u^{2i}}{2i}}.
$$
This is equal to the coefficient of
$u^{2k}$ in $(1-u^2)^{-1/2}$ and so also equal to the coefficient of  $u^k$ in
$(1-u)^{-1/2}$. This is equal to
$$
 \frac{{2k \choose k}}{4^k}.
 $$
 For the
second assertion, observe that \[ \frac{{2k \choose k}}{4^k} = \frac{1}{2}
\frac{3}{4} \cdots \frac{2k-1}{2k}.\] The third assertion follows from
Stirling's bounds \[ (2 \pi)^{\frac{1}{2}}
n^{n+\frac{1}{2}}e^{-n+1/(12n+1)} < n! < (2 \pi)^{\frac{1}{2}}
n^{n+\frac{1}{2}} e^{-n+1/(12n)}\] proved for instance on page 52 of
\cite{Fe}. \end{proof}

\begin{theorem} \label{evenkset} For $2 \leq 2k \leq n$, the probability
that an element of $S_{n}$ fixes a $2k$-set and all orbits on this invariant
subset are even is at
most $\frac{{2k \choose k}}{4^k} < \frac{1}{(\pi k)^{1/2}}$.
\end{theorem}

\begin{proof} Some subset of the cycle lengths of $\pi$ are even and add to $2k$.
(For instance if $2k=6$, then at least one of $(6),(4,2),(2,2,2)$ must
appear as cycle lengths). Using the fact that the number of permutations
with $n_i$ $i$-cycles is \[ \frac{n!}{\prod_i i^{n_i} n_i!},\] it follows
that the proportion of elements in $S_n$ fixing a $2k$-set using only even
cycles is at most
\[ \sum_{(b_2,b_4,\cdots) \atop 2b_2+4b_4+\cdots = 2k}
\sum_{(a_1,\cdots,a_n) \atop 1a_1+2a_2+\cdots = n-2k} \prod_{i \ odd}
\frac{1}{i^{a_i}a_i!} \prod_{i \ even} \frac{1}{i^{a_i+b_i}(a_i+b_i)!}.\]
Note that here $a_i+b_i$ is the number of $i$-cycles of $\pi$, and that
equality holds if $2k=n$. Since $\frac{1}{(a_i+b_i)!} \leq \frac{1}{a_i!
b_i!}$, the sought proportion is at most \[ \sum_{(b_2,b_4,\cdots) \atop
2b_2+4b_4+\cdots = 2k} \prod_{i \ even} \frac{1}{i^{b_i} b_i!}
\sum_{(a_1,\cdots,a_n) \atop 1a_1+2a_2+\cdots = n-2k} \prod_{i \geq 1}
\frac{1}{i^{a_i}a_i!}.\] Observe that \[ \sum_{(a_1,\cdots,a_n) \atop
1a_1+2a_2+\cdots = n-2k} \prod_{i \geq 1} \frac{1}{i^{a_i}a_i!} =1, \]
is the sum of reciprocals of centralizer sizes over all conjugacy
classes of the group $S_{n-2k}$. Hence the sought proportion is at most \[
\sum_{(b_2,b_4,\cdots) \atop 2b_2+4b_4+\cdots = 2k} \prod_{i \ even}
\frac{1}{i^{b_i} b_i!}, \] which is the probability that an element of
$S_{2k}$ has all cycles even; the result thus follows from Lemma
\ref{alleven}. \end{proof}

    Next we study the probability that an element $\pi \in B_n$
fixes a $k$-set using only positive cycles.

\begin{theorem} \label{reducetounitary}
\begin{enumerate}
\item The proportion of elements in $B_n$ which fix a $k$-set using only
positive cycles is at most the proportion of elements in $B_k$ with all
cycles positive.

\item The proportion of elements in $B_k$ with all cycles positive is equal
to the proportion of elements in $S_{2k}$ with all cycles even (which was
bounded in Lemma \ref{alleven}).
\end{enumerate}
\end{theorem}

\begin{proof} Recall the description of conjugacy classes and centralizer
sizes of $B_n$ given at the beginning of this subsection. If an element of
$B_n$ fixes a $k$-set using only positive cycles, then its cycle structure
vector must contain positive cycles of lengths adding to $k$. Thus the
chance that an element $\pi$ of $B_n$ fixes a $k$-set using only positive
cycles is at most
\[ \sum_{(b_1,b_2,\cdots,b_k) \atop b_1+2b_2+\cdots=k}
\sum_{(a_1,\cdots,a_n),(c_1,\cdots,c_n) \atop (a_1+c_1)+2(a_2+c_2)+\cdots
= n-k} \prod_i \frac{1}{(a_i+b_i)! c_i! (2i)^{a_i+b_i+c_i}},\] with
equality if $n=k$. Here $a_i+b_i$ is the number of positive $i$-cycles of
$\pi$ and $c_i$ is the number of negative $i$-cycles of $\pi$. Since
$\frac{1}{(a_i+b_i)!} \leq \frac{1}{a_i! b_i!}$, the sought proportion is
at most \[ \sum_{(b_1,b_2,\cdots,b_k) \atop b_1+2b_2+\cdots=k}
\frac{1}{b_i! (2i)^{b_i}} \sum_{(a_1,\cdots,a_n),(c_1,\cdots,c_n) \atop
(a_1+c_1)+2(a_2+c_2)+\cdots = n-k} \prod_i \frac{1}{a_i! c_i! (2i)^{
a_i+c_i}}.\] Observe that \[ \sum_{(a_1,\cdots,a_n),(c_1,\cdots,c_n) \atop
(a_1+c_1)+2(a_2+c_2)+\cdots = n-k} \prod_i \frac{1}{a_i! c_i! (2i)^{
a_i+c_i}} = 1,\] being the sum of reciprocals of centralizer sizes over
all conjugacy classes of the group $B_{n-k}$. Thus the sought proportion
is at most \[ \sum_{(b_1,b_2,\cdots,b_k) \atop b_1+2b_2+\cdots=k} \prod_i
\frac{1}{b_i! (2i)^{b_i}} \] which is the probability that an element of
$B_{k}$ has all cycles positive. Rewriting this sum as
\[ \sum_{(b_2,b_4,\cdots,b_{2k}) \atop 2b_2+4b_4+\cdots=2k} \prod_{i \
even} \frac{1}{b_i! i^{b_i}} \] shows that it is also the probability that
an element of $S_{2k}$ has all even cycles. \end{proof}

We let $D_n$ denote the group of signed permutations with the product of
signs equal to 1; thus $|D_n|=2^{n-1}n!$. We let $D_n^-$ denote the
nontrivial coset of $D_n$ in $B_n$, i.e. the group of signed permutations
with the product of signs equal to $-1$.

\begin{theorem} \label{reducetounitary2}
\begin{enumerate}
\item For $n>k$, the proportion of elements of $D_n$ which fix a $k$-set
using only positive cycles is equal to the proportion of elements of
$D_n^-$ which fix a $k$-set using only positive cycles. Both proportions
are at most the proportion of elements of $S_{2k}$ with all cycles even
(which was bounded in Lemma \ref{alleven}).

\item The proportion of elements of $D_k$ with all cycles positive is
twice the proportion of elements of $S_{2k}$ with all cycles even (which
was bounded in Lemma \ref{alleven}).

\end{enumerate}
\end{theorem}

\begin{proof} For part 1, note that since $n>k$, there is a bijection between the
elements of $D_n$ which fix a $k$-set using positive cycles and the
elements of $D_n^-$ which fix a $k$-set using positive cycles (just change
the sign of a cycle not involved in the $k$-set). Hence the number of
elements of $B_n$ which fix a $k$-set using positive cycles is twice the
number for $D_n$, so part 1 follows from Theorem \ref{reducetounitary}.
Part 2 follows from part 2 of Theorem \ref{reducetounitary} since elements
of $B_n$ with all cycles positive lie in $D_n$. \end{proof}

\begin{theorem} \label{1stDn}
\begin{enumerate}
\item For $n>k$, the proportion of elements in $D_n$ which fix a $k$-set
using an even (resp. odd) number of negative cycles is at most $1/2$.

\item For $n>k$, the proportion of elements in $D^-_n$ which fix a $k$-set
using an even (resp. odd) number of negative cycles is at most $1/2$.

\item For $n \geq k$, the proportion of elements in $B_n$ which fix a $k$-set
using an even (resp. odd) number of negative cycles is at most $1/2$.
\end{enumerate}
\end{theorem}

\begin{proof} From the cycle index of $B_n$, the proportion of elements
in $D_n$ which fix a $k$-set using an even number of negative cycles is at most
\begin{eqnarray*}
& & 2 \sum_{(a_1,\cdots,a_k),(b_1,\cdots,b_k) \atop \sum
i(a_i+b_i)=k, \sum b_i \ even} \sum_{(c_1,\cdots,c_n),(d_1,\cdots,d_n)
\atop \sum i(c_i+d_i)=n-k, \sum d_i \ even} \\
& &  \frac{1}{\prod_i (a_i+c_i)!
(b_i+d_i)! (2i)^{a_i+b_i+c_i+d_i}} \\
& \leq & 2 \left[ \sum_{(a_1,\cdots,a_k),(b_1,\cdots,b_k)
\atop \sum i(a_i+b_i)=k, \sum b_i \ even} \frac{1}{\prod_i a_i! b_i!
(2i)^{a_i+b_i}} \right] \\
& & \cdot \left[ \sum_{(c_1,\cdots,c_n),(d_1,\cdots,d_n)
\atop \sum i(c_i+d_i)=n-k, \sum d_i \ even} \frac{1}{\prod_i c_i! d_i! (2i)^{c_i+d_i}}
\right].
\end{eqnarray*} Here $a_i,b_i$ denote the number of positive and negative
$i$-cycles involved in fixing the $k$-set, and $c_i,d_i$ are the number of remaining
positive and negative $i$-cycles. The first term in square brackets is
the proportion of elements of $B_k$ which lie in $D_k$, which is $1/2$. The second
term in square brackets is the proportion of elements of $B_{n-k}$ which lie in $D_{n-k}$,
 which is also $1/2$. This proves the first part of the theorem. The second part is
 proved similarly. For $n>k$, part 3 is immediate from parts 1 and 2, and
 for $n=k$ part 3 follows since $D_n$ is an index two subgroup of $B_n$.
  \end{proof}

The following results about $B_n$ and $D_n$ will be useful for treating
the cases $q=2,3$.

\begin{prop} \label{7/12bound}
\begin{enumerate}
\item The proportion of elements of $B_n$ with no positive fixed
points and at most one negative fixed point is at most $7/12$.

\item The proportion of elements of $B_n$ ($n \geq 2$) with at most one
positive fixed point and at most one negative fixed point is at most
$5/6$.
\end{enumerate}
\end{prop}

\begin{proof} From the cycle index of the groups $B_n$, the proportion of elements
with no positive fixed point and at most one negative fixed point is the
coefficient of $u^n$ in $\frac{(1+u/2)}{(1-u)e^u}$. This is easily seen to
be at most $7/12$, this value being attained at $n=3$. For the second
assertion, one uses the generating function $\frac{(1+u/2)^2}{(1-u)e^u}$,
and the coefficient of $u^n$ is at most $5/6$, this value being attained
at $n=3$.
\end{proof}

We need one more result about elements in $B_n$.

\begin{theorem} \label{Bevencycle2}
\begin{enumerate}
\item The $n \rightarrow \infty$ proportion of elements in $B_n$ which
fix a $k$-set using an even number of negative cycles, and have at most one
positive fixed point and at most one negative fixed point is at most  \[
\frac{9}{8e} \leq .414.\]

\item  The $n \rightarrow \infty$ proportion of elements in $B_n$ which
fix a $k$-set using an odd number of negative cycles, and have at most one
positive fixed point and at most one negative fixed point is at most  \[
\frac{9}{8e} \leq .414.\]

\end{enumerate}
\end{theorem}

\begin{proof} From the cycle index of $B_n$, the proportion of elements
in $B_n$ which fix a $k$-set using an even number of negative cycles, and
have at most one positive fixed point and at most one negative fixed point
is at most
\begin{eqnarray*}
& & \sum_{(a_1 \leq 1,\cdots,a_k),(b_1 \leq 1,\cdots,b_k) \atop \sum
i(a_i+b_i)=k, \sum b_i \ even} \sum_{(c_1 \leq 1,\cdots,c_n),(d_1 \leq
1,\cdots,d_n) \atop \sum i(c_i+d_i)=n-k} \\
& &  \frac{1}{\prod_i (a_i+c_i)!
(b_i+d_i)! (2i)^{a_i+b_i+c_i+d_i}} \\
& \leq & \left[ \sum_{(a_1 \leq 1,\cdots,a_k),(b_1 \leq 1,\cdots,b_k)
\atop \sum i(a_i+b_i)=k, \sum b_i \ even} \frac{1}{\prod_i a_i! b_i!
(2i)^{a_i+b_i}} \right] \\
& & \cdot \left[ \sum_{(c_1 \leq 1,\cdots,c_n),(d_1 \leq 1,\cdots,d_n)
\atop \sum i(c_i+d_i)=n-k} \frac{1}{\prod_i c_i! d_i! (2i)^{c_i+d_i}}
\right].
\end{eqnarray*} Here $a_i,b_i$ denote the number of positive and negative
$i$-cycles involved in fixing the $k$-set, and $c_i,d_i$ are the number of remaining
positive and negative $i$-cycles.

The first sum in square brackets is (by the cycle index of $B_k$) equal to
the proportion of elements in $B_k$ with an even number of negative
cycles, at most one positive fixed point, and at most one negative fixed
point; such elements lie in $D_k$ so the proportion is at most $1/2$. Thus
the proportion of elements in $B_n$ which fix a $k$-set using an even number
of negative cycles, and have at most one positive fixed point and at most
one negative fixed point is at most $1/2$ multiplied by the proportion of
elements of $B_{n-k}$ with at most one positive fixed point and at most
one negative fixed point. By Theorem \ref{hyperoctPoisson}, the $n
\rightarrow \infty$ limiting proportion of elements of $B_n$ with at most
one positive fixed point and at most one negative fixed point is equal to
$\left[ \frac{(1+1/2)}{e^{1/2}} \right]^2 = \frac{9}{4e}$, which proves
part 1 of the theorem. Part 2 is proved by the same reasoning.
\end{proof}

Theorem \ref{2ndDn} gives a type $D$ analog of Theorem \ref{Bevencycle2}.

\begin{theorem} \label{2ndDn}
\begin{enumerate}
\item The $n \rightarrow \infty$ proportion of elements of $D_n$ (or $D_n^-$) with
at most one positive fixed point, at most one negative fixed point, and
which fix a $k$-set using an even number of negative cycles is at most
$\frac{9}{8e} \leq .414$.

\item The $n \rightarrow \infty$ proportion of elements of $D_n$ (or $D_n^-$) with
at most one positive fixed point, at most one negative fixed point, and
which fix a $k$-set using an odd number of negative cycles is at most
$\frac{9}{8e} \leq .414$.

\end{enumerate}
\end{theorem}

\begin{proof} Since $k$ is fixed, we can assume that $n \geq k+3$. Then
there is a bijection between elements in $D_n$ with at most one positive
fixed point, at most one negative fixed point and which fix a $k$-set
using an even number of negative cycles, and elements in $D_n^-$ with the
same restrictions. Indeed, one can switch the sign of a cycle of length
$\geq 2$ which is not involved in the $k$-set. Hence the proportions in
part 1 are equal to the corresponding proportions for $B_n$, and part 1 is
immediate from Theorem part 1 of Theorem \ref{Bevencycle2}. Similarly part 2
follows from part 2 of Theorem \ref{Bevencycle2}. \end{proof}

\begin{theorem} \label{newthe3}
\begin{enumerate}
\item The $n \rightarrow \infty$ proportion of elements of $B_n$ with no positive fixed points, at
most one negative fixed point, and which fix a $k$-set using an even number of negative cycles is
at most $.276$.

\item  The $n \rightarrow \infty$ proportion of elements of $B_n$ with no positive fixed points, at
most one negative fixed point, and which fix a $k$-set using an odd number of negative cycles is
at most $.276$.

\item The $n \rightarrow \infty$ proportion of elements of $D_n$ (or $D_n^-$) with no positive fixed points, at
most one negative fixed point, and which fix a $k$-set using an even number of negative cycles is
at most $.276$.

\item  The $n \rightarrow \infty$ proportion of elements of $D_n$ (or $D_n^-$) with no positive fixed points, at
most one negative fixed point, and which fix a $k$-set using an odd number of negative cycles is
at most $.276$.
\end{enumerate}
\end{theorem}

\begin{proof} For parts 1 and 2, arguing as in Theorem \ref{Bevencycle2} shows that the sought proportion is
at most $1/2$ times the $n \rightarrow \infty$ limiting proportion of elements of $B_n$ with no
positive fixed points and at most one negative fixed point. By Theorem \ref{hyperoctPoisson}, this is equal to
\[ \frac{1}{2} \left[ \frac{1}{e^{1/2}} \frac{1}{e^{1/2}} (1+1/2) \right] \leq .276.\]
For parts 3 and 4, one argues as in Theorem \ref{2ndDn} to reduce to the $B_n$ case.
\end{proof}

\section{Maximal tori and the Weyl group} \label{maximaltori}

For $G$ a finite classical group, this section gives upper bounds  on  the proportion
of elements of $G$ which are regular semisimple and fix a $k$-space in
terms of the proportion of elements of the Weyl group $W$ which fix a
$k$-set. Since finite classical groups contain many regular semisimple
elements (this is made more precise in Section \ref{regularsemisimple}),
this will enable us to bound away from $0$ the proportion of elements of
$G$ which are regular semisimple and derangements on $k$-spaces.

Let $X$ be a simple algebraic group over an algebraically closed field
of positive characteristic $p$ defined over the prime field.  Let $\sigma_q$ denote
the  Frobenius
endomorphism with fixed points $X(q)$.  Let $\sigma = \sigma_q$ or
$\sigma_q \tau$ where $\tau$ is a graph automorphism.
We only consider the case  the graph automorphism $\tau$ is trivial or has
order $2$  (since
we are only dealing with classical groups and we may assume the rank
is large, this is not a problem).  Let $G=X_{\sigma}$, the set of fixed points of
$\sigma$.  This is a finite group of Lie type defined over the field of $q$ elements.

Let $W$ denote the Weyl group of $X$ and $W_0: = \langle W, \tau \rangle$
the extended Weyl group (i.e. the normalizer of a maximal torus $T$ of $X$
in $\langle X, \tau \rangle$  --- we may choose $\tau$ to normalize some
$\sigma$-invariant maximal torus $T$).   In order to state the results
uniformly, we view $\tau =1$ if the graph automorphism is not present.

There is a bijection between conjugacy classes of $W_0$ in the coset
$\tau W$  and
conjugacy classes of maximal tori  -- if  $w \in W$, let $T_w$ denote the
corresponding maximal torus in G  (up to conjugacy)  --- see \cite{SpSt}
for the basic background on this.

We say that $T_w$ is nondegenerate
if the centralizer of $T_w$ in the algebraic group is a (maximal) torus.  Let
$N_w$ denote the normalizer of $T_w$.   If $T_w$ is nondegenerate,
then $N_w/T_w \cong C_W(w)$.  In any case,  $|N_w/T_w| \ge |C_W(w)|$.
 If $x \in G$ is regular semisimple, then $x$ is in a unique maximal torus
(its centralizer).   Also, if $T_w$ contains a regular semisimple element,
then it certainly is nondegenerate.

  Choose a set of representatives $R$ for the conjugacy classes in
  the coset $\tau W_0$.
  If   $S$  is a subset of $R$,
  let $G_S$ denote the set of  semisimple elements
of G conjugate to an element of $T_w$ for some $w \in S$.    So $G_S$ is the
union of all the conjugates of $T_w$,  for $w \in X$.     Note that the union
of conjugates of $T_w$  has size at most $[G:N_w]|T_w| = |G|/|C_W(w)|$.

Thus, \[ |G_S| /|G|   \le    |G|^{-1}  \sum_{w \in S}    [G:N_w]|T_w|   =
\sum_{w \in S}  |C_W(w)|^{-1} = \sum_{w \in S} |W|^{-1} |w^W| \]  is equal to the
proportion of elements of $W$  conjugate to an element of $S$.

If we want to estimate the proportion of regular semsimple elements
conjugate to an element of $T_w, w \in S$, it suffices to sum
over those $T_w$ which contain a regular semisimple element
and so improve the estimate (for $q$ sufficiently large with $G$
of fixed rank,  all maximal tori
will contain semisimple regular elements).

 Let $G$ be a classical group over a finite field with natural module $V$.
 Let $U$ be either a totally singular or nondegenerate subspace of $V$.
 Suppose that $x \in G$ is regular semisimple with $T$ the maximal torus
 of $G$ containing $x$.

 Note that if $G=SL,Sp$ or $SU$, then $x$ has distinct eigenvalues on $V$,
 whence $x$ and $T$ have precisely the same invariant subspaces.
 If  the stabilizer of $U$ is connected
 (in the algebraic group), then any semisimple element stabilizing $W$
 is in a maximal torus of $U$.   In particular, this holds if the characteristic
 is $2$ or $U$ is totally singular.  So in those cases, $x$ leaves $U$
 invariant if and only if $T$ does.

 Finally, assume that the characteristic is not $2$,  $U$ is nondegenerate
 of even dimension  and $G=SO^{\epsilon}(n,q)$.
 Let $U'=U^{\perp}$.  The connected part of the stabilizer of $U$ is $SO(U)
 \times SO(U')$.   So if $\det(x|_U)=1$, then $T$ preserves $U$ as well.
 The only other possibility is that $\det(x|_U)=-1$.  Indeed, in this case $T$
 need  not leave $W$ invariant.   This forces $x$ to have $-1$ as an eigenvalue
 on each of $U$ and $U'$.

 We count these elements separately.  If $q$ is large, it is easy to show that the proportion
 of such elements goes to $0$ with $q \rightarrow \infty$ (uniformly in $n$).   We can also
 observe that such an $x$ will fix a nondegenerate space $U''$ with $\dim U= \dim U''$
 (by interchanging a $-1$ eigenspace and a $1$ eigenspace) with $\det x|_{U''} =1$.
 Thus, $x$ will be in a maximal torus fixing a nondegenerate space of the given dimension
 (of one type or the other).   Moreover, if $T_w$ is a maximal torus containing such an element $x$,
 then $w$ must have a fixed point and indeed if $n$ is even, $w$ must have two fixed points.

\begin{theorem} \label{Gltorus} Suppose that $1 \leq k \leq n/2$.
The proportion of elements of any subgroup $H$ between $SL(n,q)$ and $GL(n,q)$
which are regular semisimple and fix a $k$-space is at most the proportion
of elements in $S_n$ which fix a $k$-set. In fact, it is at most the
proportion of elements in $S_n$ which fix a $k$-set and have at most $q-1$
fixed points.
\end{theorem}

\begin{proof} Let $w \in S_n$ and let $T_w$ be the corresponding
maximal torus of $H$.   If $w$  has cycles of length  $a_1, \ldots, a_r$, then $V = \oplus
V_i$ where $\dim V_i =a_i$ and $T_w$ acts irreducibly on each $V_i$.  As
long as $T_w$ contains a regular semisimple element, then these are the
only subspaces left invariant by $T_w$ (and so by any regular semisimple
element in $T_w$ as well, because of the uniqueness of $T_w$).

Thus,  $x$ regular semisimple in $T_w$ fixes a $k$-space if and only if $w$
fixes a $k$-set (i.e. some subset of the $a$'s adds to $k$).

Note that for a fixed $q$, the number of $1$-cycles in $w$ is at most $q-1$  or
$T_w$ will not contain any regular semisimple elements.
\end{proof}

\begin{theorem} \label{Utorus}

\begin{enumerate}
\item For $1 \leq k \leq n/2$, the proportion of elements in any coset of
$SU(n,q)$ in $U(n,q)$ which are regular semisimple and fix a
nondegenerate $k$-space is at most the proportion of elements in $S_n$
which fix a $k$-set.

\item For $1 \leq k \leq n/2$,
the proportion of elements in any coset of $SU(n,q)$ in $U(n,q)$
which are regular semisimple and fix a totally singular $k$-space is at
most the proportion of elements in $S_n$ which fix a $2k$-set, and
have all orbits in this invariant subset even.
\end{enumerate}
\end{theorem}

\begin{proof} For $SU(n,q)$,   $\langle \tau, W \rangle \cong S_n \times  \Z/2$ and so we are really
considering conjugacy classes  in $S_n$.

Consider $w$ having cycles of lengths  $a_1, \ldots, a_r$.
Then $T_w$ acts on $V_1  \perp \ldots \perp V_r$ where $\dim V_i =a_i$.

If $a_i$ is odd, then $T_w$ is irreducible on $V_i$ and if $a_i$ is
even, then $T_w$ leaves invariant precisely two proper subspaces each
totally singular of dimension $a_i/2$.

Thus, the probability of being regular semisimple and fixing a nondegenerate $k$-space
is at most the probability of fixing a $k$-set.   Similarly a regular semisimple
element fixing a totally singular $k$-space lies in $T_w$ where
$w$ leaves invariant a $2k$-set with all orbits of $w$ on this set of even length.
The results follow.
\end{proof}

We let $B_n$ denote the hyperoctahedral group of signed permutations on
$n$ symbols.

\begin{theorem} \label{Sptorus1}
\begin{enumerate}
\item For $1 \leq k \leq n$, the proportion of elements in the group $Sp(2n,q)$
which are regular semisimple and fix a nondegenerate $2k$-space is at most
the proportion of elements in $S_n$ which fix a $k$-set.

\item For $1 \leq k \leq n$, the proportion of elements in $Sp(2n,q)$ which
are regular semisimple and fix a totally singular $k$-space is at most the
proportion of elements in $B_n$ which fix a $k$-set using only positive
cycles.

\item Consider the proportion of regular semisimple
elements in $Sp(2n,q)$. For $q=2$ it is at most the proportion of elements
in $B_n$ with no positive fixed points and at most one negative fixed
point. For $q=3$ it is at most the proportion of elements in $B_n$ with at
most one positive fixed point and at most one negative fixed point.
\end{enumerate}

\end{theorem}

\begin{proof} View $w$ in $W$ as $(a_1, \epsilon_1), \ldots, (a_r,
\epsilon_r)$ where $(a_1, \ldots, a_r)$ is a partition of $n$ and
$\epsilon_i = \pm 1$. Let ${\bar w}$ denote the image of $w$ in $S_n$  (so
it has cycles of size $a_1, ..., a_r$).

Then $V = V_1 \perp \ldots \perp V_r$ where $\dim V_i=2a_i$  and if
$\epsilon_i=-$, then $T_w$ acts irreducibly on $V_i$ while if $\epsilon_i
= +$, then $V_i = A \oplus B$ where A and B are totally singular subspaces
of $V_i$ with $T_w$ acting irreducibly on A and B (with A and B
nonisomorphic as $T_w$-modules). Thus,  the only nondegenerate subspaces
left invariant by the $T_w$ are sums of $V_j$ for some subset.   Thus,  the proportion
of elements which are both semisimple and leave invariant a nondegenerate
subspace of dimension $2k$ is at most the probability that
a random element of $S_n$ fixes a $k$-set.

For totally singular $k$-spaces (so $k \le n$),  we would need some subset
of the $V_i$ corresponding to $\epsilon_i=+$ with the $a_i$ adding up to
$k$. In particular, some subset of the $a_i$ must add up to $k$ and so we get
the same upper bound (or just using positive cycles if we want a somewhat
better bound).

Suppose that $T_w$ contains a regular semisimple element.  If $q =2$,
there can be no positive 1-cycles (for then $T_w$ is trivial on a
2-dimensional space) and there can be at most 1 negative 1-cycle (otherwise
either $(T-1)^2$ or $(T^2+T + 1)^2$ divides the characteristic polynomial
of any element of $T_w$).

If $q = 3$,  similarly we see that there can be at most 1 positive $1$-cycle
(using the fact that any element has  $\det = 1$  -- if we are in the
conformal symplectic group, then there can be at most $2$ positive cycles of
length 1).  Similarly, there can be at most 1 negative cycle of length 1
(otherwise $(T-1)^2$, $(T+1)^2$ or $(T^2+1)^2$ divides the characteristic
polynomial of any element of $T_w$).
\end{proof}

The next result will be useful in analyzing the action of $Sp(2n,q), q$ even on
nondegenerate hyperplanes in the $2n+1$ dimensional orthogonal representation.   Note that the
stabilizers of these hyperplanes are orthogonal groups.

\begin{theorem} \label{Sptorus2} Let $q$ be even.
The proportion of elements in $Sp(2n,q)$ which are regular semisimple
and fix a positive (resp. negative) type nondegenerate hyperplane is at
most $1/2$.
\end{theorem}

\begin{proof}
 We view $\Omega(2n+1,q)=Sp(2n,q)=G$.   Suppose that $T_w$ is a nondegenerate
 maximal torus of $G$.  It is well known that every element of $G$ is contained
 in a conjugate of     $O^+(2n,q)$
or $O^-(2n,q)$.  However, a  regular semisimple element is in precisely $1$ (since a regular semisimple
element of $Sp(2n,q)$ has no eigenvalue $1$).    Similarly, we see that each nondegenerate
maximal torus stabilizes a unique nondegenerate hyperplane.

It is straightforward to see that
the $T_w$ in $\Omega^+$ are those such that $w$ is in the
type $D$ subgroup of index $2$ in the Weyl
group, and the
rest are in $\Omega^-$.

Let $R$ be a set of representatives for the $W$ conjugacy classes of elements
contained in the subgroup $D$ of index $2$ in $W$.
So we obtain an upper bound for the number $N$ of regular semisimple elements in $O^+$
by noting that:

\begin{eqnarray*}
N  & \le & \sum_{w \in R} [G:N(T_w)]  |T_w| \\
& = & \sum_{w \in R} |G| |N(T_w):T_w|^{-1} \\
& \le &  |G| \sum_{w \in R}  |C_W(w)|^{-1} \\
& =  &  |G| \sum_{w \in D} |W|^{-1} \\
& = & |G||D|/|W| \\
& = & (1/2)|G|. \end{eqnarray*}

Thus the probability that a random element of $Sp(2n,q)$ is both regular  semisimple
and is in $O^+$ is at most $1/2$ and similarly for $O^-$  (we just get
those $T_w$ with $w$ not in the subgroup of index 2)  -- and in the limit
as $q$ increases, the proportion approaches $1/2$.
\end{proof}

For orthogonal groups, regular semisimple does not imply distinct eigenvalues.
We say an element in an $n$-dimensional orthogonal group is {\it strongly regular semisimple} if it has
$n$ distinct eigenvalues.  Note that if $x$ is strongly regular semisimple and
$x \in T$, a maximal torus, then a subspace is $x$-invariant if and only if it is $T$-invariant.
If $x$ is regular semisimple, the same statement is true for totally singular spaces.

\begin{theorem} \label{Otorus1} Let $q$ be odd.   Let  $G =\Omega(2n+1,q)$.
\begin{enumerate}
\item The proportion of elements of $G$ which are strongly regular
semisimple and fix a nondegenerate positive (resp. negative) type space
of dimension $2k$   is at most the proportion of elements in $B_n$
which fix a $k$-set using an even (resp. odd) number of negative cycles.
If $q=3$ the elements in $B_n$  can have no positive fixed points and
at most one negative fixed point.

\item The proportion of elements of $G$ which are semisimple and fix a nondegenerate
$2k$-space is at most the proportion of elements of $S_n$ which fix a $k$-set.

\item The proportion of elements of $G$ which are regular
semisimple and fix a totally singular space of dimension $k$ is at most
the proportion of elements in $B_n$ which fix a $k$-set using only
positive cycles.

\end{enumerate}
\end{theorem}

\begin{proof}
Again, the Weyl group is the group of signed permutations.
Let $w$ be an element of the Weyl group with
corresponding torus $T_w$. Then $V=V_0 \perp V_1 \perp \ldots \perp V_r$
where $\dim(V_0)=1$, and $\dim(V_i)=2a_i$ with $T_w$ acting irreducibly
on $V_i$ if $\epsilon_i= -$ and acting irreducibly on a pair of
totally isotropic subspaces of $V_i$ if $\epsilon_i=+$ (the type
of $V_0$ is determined by the other $V_i$). We argue exactly as in the proof
Theorem \ref{Sptorus2} (note that if the nondegenerate space has
dimension $2k$, we need that $\sum a_i=k$ for some subset of the $a$'s,
i.e. the image of $w$ in the symmetric group preserves a $k$-set).

If $q=3$ and the maximal torus $T_w$ contains regular semsimple elements
then  $w$ has at most one fixed point of each sign. Similarly, if $q=3$ and
$T_w$ contains strongly regular semisimple elements, then $w$ has no positive
fixed points and at most one negative fixed point.

For part (2), let $x$ be semisimple. Suppose that $x$ fixes a nondegenerate $2k$-space.
Then either  $x$ is contained in a maximal torus $T_w$ which fixes that $2k$-space or
$\det(x)=-1$ on the $2k$-space (and so on the complement as well).

Note that this implies that $x$ has eigenvalues $\pm 1$ on the $2k$-space
and an eigenvalue $-1$ on the orthogonal complement.  Choosing a different
nondegenerate $2k$-space that is $x$-invariant where $\det(x) =1$ (by swapping the
$1$-eigenvector and a $-1$ eigenvector) shows that $x$ is conjugate to an element of $T_w$,
where $w$ has cycle sizes adding up to $k$, whence the result.

Suppose that a nondegenerate maximal torus $T_w$ fixes a totally singular $k$-space.
Then $\sum a_j =k$ for some subset of the
$a_j$ all of positive type   (ignoring the positivity, we get an upper  bound of the
proportion  of elements in $S_n$ fixing a $k$-set).
\end{proof}

We let $D_n$ denote the group of signed permutations with the product of
signs equal to $1$; thus $|D_n|=2^{n-1} n!$.

\begin{theorem} \label{Otorus2}
 \begin{enumerate}

\item Let $q$ be odd. The proportion of elements of the group $\Omega^+(2n,q)$ which are strongly regular
semisimple and fix a nondegenerate positive (resp. negative) type
$2k$-space is at most the proportion of elements of $D_n$ which fix a
$k$-set using an even (resp. odd) number of negative cycles. If $q=3$, the element of $D_n$ has no positive
fixed points and at most one negative fixed point.

\item Let $q$ be even. The proportion of elements of $\Omega^+(2n,q)$ which are regular
semisimple and fix a nondegenerate positive (resp. negative) type
$2k$-space is at most the proportion of elements of $D_n$ which fix a
$k$-set using an even (resp. odd) number of negative cycles. If $q=2$, the
element of $D_n$ has at most one positive fixed point, at most one negative
fixed point, no positive 2-cycles, and at most one negative 2-cycle.
If $q=4$, the element of $D_n$ has at most two positive fixed points.

\item   The proportion of elements of $\Omega^+(2n,q)$ which are
 semisimple and fix a nondegenerate
$2k$-space is at most the proportion of elements of $S_n$ which fix a
$k$-set.

\item The proportion of elements of $\Omega^+(2n,q)$ which are regular
semisimple and fix a totally singular $k$-space is at most the proportion
of elements of $D_n$ which fix a $k$-set using only positive cycles.

\end{enumerate}
\end{theorem}

\begin{proof} So the Weyl group has order $2^{n-1}n!$,  i.e. signed permutations with
the product of the signs 1.
If $w$ corresponds to $(a_1, \epsilon_1), ..., (a_r,\epsilon_r)$, then
$V = V_1 \perp \ldots \perp V_r,  \dim V_i = 2a_i$, where $T_w$ is
irreducible on $V_i$ if $\epsilon_i=-$ and preserves a pair of totally singular
spaces if $\epsilon_i=+$.

So $T_w$ preserves a nondegenerate 2k space if and only if $\sum a_j = k$
for some subset, i.e. ${\bar w}$ fixes a $k$-set.

Note that if $g$ is strongly semisimple regular and it fixes a nondegenerate
$2k$-space, then  the maximal torus $T$ containing $g$ will also fix
that space. If $q=3$ and $T_w$ contains strongly regular
semisimple elements, $T_w$ has no positive fixed points and at most one negative
fixed point.  Now (1) follows.

Suppose that $q=2$.  Then $T_w$ containing semisimple regular elements implies
that $w$ has at most one fixed point of each type, no positive $2$-cycles
and at most $1$ negative $2$-cycle (corresponding to eigenvalues of order $5$).
Similar observations yield the statement about $q=4$ and so (2) follows.

If $g$ is just semisimple and fixes a nondegenerate $2k$-space,
it is not hard to see that $g$ fixes some nondegenerate $2k$-space
such that $\det g = 1$ on that $2k$-space.  Thus, $g$ is contained
in a maximal torus fixing the second $2k$-space, whence
$g$ is a conjugate to
an element of $T_w$ where $T_w$ fixes some nondegenerate $2k$-space
(but not necessarily of the same type).  Thus,  $w$ will fix a $k$-subset
and (3) follows.

$T_w$ will preserve a totally singular $k$-space if and only if some subset
of the positive cycles add up to k, proving (4).

\end{proof}

In the next result, we let $D_n^-$ denote the nontrivial coset of $D_n$ in
$B_n$. Thus $D_n^-$ is the set of signed permutations with the product of
signs equal to $-1$, and $|D_n^-|=2^{n-1} n!$.

\begin{theorem} \label{Otorus3}
\begin{enumerate}

\item Let $q$ be odd. The proportion of elements of the group $\Omega^-(2n,q)$ which are strongly regular
semisimple and fix a nondegenerate positive (resp. negative) type
$2k$-space is at most the proportion of elements of $D_n^-$ which fix a
$k$-set using an even (resp. odd) number of negative cycles. If $q=3$, the element of $D_n^-$ has no positive
fixed points and at most one negative fixed point.

\item Let $q$ be even. The proportion of elements of $\Omega^-(2n,q)$ which are regular
semisimple and fix a nondegenerate positive (resp. negative) type
$2k$-space is at most the proportion of elements of $D_n^-$ which fix a
$k$-set using an even (resp. odd) number of negative cycles. If $q=2$, the
element of $D_n^-$ has at most one positive fixed point, at most one negative
fixed point, no positive 2-cycles, and at most one negative 2-cycle.
If $q=4$, the element of $D_n^-$ has at most two positive fixed points.

\item   The proportion of elements of $\Omega^-(2n,q)$ which are
 semisimple and fix a nondegenerate
$2k$-space is at most the proportion of elements of $S_n$ which fix a
$k$-set.

\item The proportion of elements of $\Omega^-(2n,q)$ which are regular
semisimple and fix a totally singular $k$-space is at most the proportion
of elements of $D_n^-$ which fix a $k$-set using only positive cycles.

\end{enumerate}
\end{theorem}

\begin{proof} The analysis is the same as in Theorem \ref{Otorus2}.
\end{proof}

\section{Large Fields} \label{Large Fields}

We prove our main result in the case that $q$ is large.    We first note
that by \cite{FNP, GL} it follows that the proportion of regular semisimple
elements in a classical group is at least $1 - O(1/q)$, where the implied
constant is absolute.    Indeed, the proof in \cite{GL} actually shows the following:

\begin{theorem} \label{GL result}  Let $S$ be a finite simple group of Lie type
over a field of size $q$.  Let $g$ be an inner diagonal automorphism of $S$.
The proportion of semisimple regular elements in the coset $gS$ is
at least $1 - O(1/q)$.   In particular, the proportion of regular semisimple
elements in $gS$ goes to $1$ as $q \rightarrow \infty$.
\end{theorem}

Note that the same result applies if we replace
$S$ by a quasisimple group.  The $O(1/q)$ error term is given quite explicitly
in \cite{GL}.    One can give an alternate proof of this result by using generating
functions.  We next extend this to strongly semisimple regular elements in
orthogonal groups.

\begin{theorem}  \label{strongly regular}  Let $S=\Omega^{\pm}(d,q)=\Omega^{\pm}(V)$.
 Let $g$ be an inner diagonal automorphism of $S$.  The proportion of elements in
 the coset $gS$ which are not strongly regular semisimple is at most $O(1/q)$.
\end{theorem}

\begin{proof}  By the previous result, it suffices to consider semisimple regular
 elements which are not  strongly regular semisimple.   Let $x$ be such an element.
Then $x$ is conjugate to an element of the subgroup
 $J:=SO(W) \times SO(W^{\perp})$ where $x$ acts as $\pm 1$ on the nondegenerate
 $2$-space $W$.
The union of the conjugates of all such elements for a fixed $W$ has size
at most $2|S|/(q-1)$  (because this set is invariant under $J$).  There
are two different
choices for $W$ (either it has $+$ type or $-$ type).   Thus, the
proportion of  elements
in $gS$ which are
semisimple regular elements with a $2$-dimensional eigenspace
is at most $4/(q-1)$.
The result follows.
\end{proof}

We next  deal with a special case.

\begin{theorem} \label{thm:  nondeg1}
Let $G =\Omega^{\pm}(2n,q) = \Omega^{\pm}(V)$.   The probability that
$g \in G$ fixes a nondegenerate space of odd dimension $k$
is at most $O(1/q)$ (the implied constant is independent of $n, k$).
\end{theorem}

\begin{proof}  Suppose that $x \in SO(V)$ is semisimple and fixes a
nondegenerate $k$-space with
$k$ odd.  Then $x$ has an eigenvalue $\pm 1$ with multiplicity at least $2$
(and exactly $2$ if $x$ is semisimple regular).   In particular, $x$
is not strongly
semisimple regular. Apply the previous result. \end{proof}

\begin{theorem} \label{thm: large q} Let $G$ be a classical group over a field
of $q$-elements with natural module of dimension $n$.    Fix a positive integer
$k \le n/2$.
\begin{enumerate}
\item   The probability that a random element of $G$ fixes a nondegenerate
subspace of dimension $k$ is at most $(2/3)+ O(1/q)$.
\item   The probability that a random element of $G$ fixes a totally singular
subspace of dimension $k$ is at most $(1/2) + O(1/q)$.
\item   Let $\epsilon > 0$.   There exists $N > 0$    such that
if $n \ge 2k > N$, the probability that a random element of $G$ fixes a totally singular
or nondegenerate subspace of dimension $k$ is less than $\epsilon + O(1/q)$.
\end{enumerate}
\end{theorem}

\begin{proof}   By the remarks above, it suffices to compute the probability
that a regular semisimple element fixes the corresponding type of subspace.
By the previous section, this is bounded above by the proportion of elements
in the Weyl group conjugate to a subgroup (or in the twisted cases a coset).
Now apply the results of Sections \ref{alternating}, \ref{otherWeyl}
and \ref{maximaltori} .
\end{proof}

Note that the same proof implies the same result for elements in each
coset of the corresponding quasisimple group.  In particular, we see that  as $q, k
\rightarrow \infty$, the proportion of derangements goes to $1$.

Note also that the $O(1/q)$ comes in only in estimating the proportion of elements
which are not regular semisimple.    In particular, it follows that:

\begin{cor} \label{weak asymptotic}  Let $G$ be a classical group over a field
of $q$-elements with natural module of dimension $n$.    Fix a positive integer
$k \le n/2$.    Let $\epsilon > 0$.   There exists $N > 0$    such that
 if $n \ge 2k > N$, the probability that a random element of $G$ is both regular
 semisimple and fixes a totally singular
or nondegenerate subspace of dimension $k$ is less than $\epsilon$.
\end{cor}

From the above discussion we can easily deduce the
Boston--Shalev conjecture for subspace actions unless both $q$ and the
dimension of the subspace are fixed (and the dimension of the natural module is growing).
Indeed, if $q$ is large one applies Theorem \ref{thm: large q}. If $q$ is fixed
and $k$ is growing, one uses Corollary \ref{weak asymptotic} and the fact
that the proportion of regular semisimple elements is bounded away from $0$
(which follows by \cite{FNP}).

\section{Regular semisimple elements} \label{regularsemisimple}

    This section discusses estimates on the proportions of regular semisimple
elements in the simple classical groups.
In view of the results of the previous section, the case of $q$ fixed
is the critical case.   The main results of this section are
exact formulas for the fixed $q$, $n \rightarrow \infty$ limiting proportion
of regular semisimple elements in the groups $GL,U,Sp,\Omega^{\pm}$ (the
case of  $GL$ is due independently to \cite{Fu2},\cite{W2} and the cases
$U,Sp$ are in \cite{FNP}). The case of $\Omega$ requires new ideas.

We also discuss some closely related results about the asymptotic equidistribution
of regular semisimple elements in cosets $gH$ of a simple classical group in
larger groups $G$. This gives a different approach to some results of
Britnell \cite{B1},\cite{B2} (derived using generating functions); our
approach has the advantage of generalizing to $\Omega$.

    The proportion of regular semisimple elements has also been studied in \cite{FlJ};
however their formulae do not seem easily suited to asymptotic analysis.

\subsection{SL}

    Theorem \ref{Glregss} gives the fixed $q$, large $n$ limiting proportion of
regular semisimple elements in $GL(n,q)$. This result will be crucial in
this paper.

\begin{theorem} \label{Glregss} (\cite{Fu2},\cite{W2}) The fixed $q$, $n \rightarrow
\infty$ limiting proportion of regular semisimple elements in $GL(n,q)$ is
$1-1/q$. \end{theorem}

    Using algebraic geometry Guralnick and L\"{u}beck established the following
result (which can be proved less conceptually using generating functions).

\begin{theorem} (\cite{GL}) The proportion of regular semisimple elements
in $PSL(2,q)$ is at least $1-(2,q-1)/(q-1)$. For $n \geq 3$, the
proportion of regular semisimple elements in $PSL(n,q)$ is at least $1-
1/(q-1)-2/(q-1)^2$. In particular, as $q \rightarrow \infty$, these
proportions go to 1 uniformly in $n$.
\end{theorem}

The same is true for any given coset of $PSL$ in $PGL$.
(with the same
proof). The following result was proved by Britnell using
generating functions.

\begin{theorem} \label{BritnellSL} (\cite{B1}) The fixed $q$, large $n$
limiting
 proportion of regular semisimple elements in any coset of $SL(n,q)$
 in $GL(n,q)$ is equal to $1-1/q$, which is the corresponding limit for
 $GL(n,q)$. Furthermore the same holds for a $GL(n,q)$ coset of any subgroup
 $H$ between $SL(n,q)$ and $GL(n,q)$. \end{theorem}

Although generating function methods lead to the most precise bounds, we use
the relationship between maximal tori and the Weyl group (Section \ref{maximaltori}) to obtain a different proof
of Theorem \ref{BritnellSL}. This will also be useful in reducing the
study of regular semisimple derangements from a coset of $SL(n,q)$ in
$GL(n,q)$ to $GL(n,q)$ (see Theorem \ref{eigenequal}). Also, this approach will be useful in studying
$\Omega$ in odd characteristic (where generating function methods seem difficult).

\begin{theorem} \label{differenceSL} Let $H$ be a subgroup between $SL(n,q)$ and
$GL(n,q)$. The difference in the proportion of regular semisimple elements
in any two cosets of $H$ in $GL(n,q)$ is at most $\frac{c_q
\log(n)^3}{n^{1/2}}$, where $c_q$ is independent of $n$.
\end{theorem}

\begin{proof} One can suppose that $H=SL(n,q)$ since cosets of other subgroups are
unions of cosets of $SL(n,q)$   and
we are supposing $q$ is fixed. Let $T_w$ be a
maximal torus of $GL(n,q)$. Suppose that $w$ has $r$ distinct cycle
lengths $a_1,\cdots,a_r$ occurring with multiplicities $m_1,\cdots,m_r$.

    We claim that if $\gcd(a_1m_1,\cdots,a_rm_r,q-1)=1$, then the number of
regular semisimple elements of $T_w$ is the same for each coset. To see
this, choose scalars $c_1,\cdots,c_r$ such that $\prod c_i^{a_im_i}=\zeta$
where $\zeta$ is a generator of the multiplicative group of
$\mathbb{F}_q^*$ (the gcd condition guarantees that this is possible).
Write $t \in T_w$ as $(t_1,\cdots,t_r)$ where the $t_i$ correspond to the
cycles of length $a_i$ (and so there are $m_i$ blocks). Then define a map
$T_w \mapsto T_w$ by sending $(t_1,\cdots,t_r)$ to
$(c_1t_1,\cdots,c_rt_r)$. This multiplies the determinant by $\zeta$ and
is a bijection on regular semisimple elements. (All we need to verify
is that if  $(t_1,\cdots,t_r)$ is regular semisimple so is $(c_1t_1,\cdots,c_rt_r)$.
Since the minimal polynomials of distinct $t_i$ have all factors of
degree $a_i$, it is clear that $c_it_i$ and $c_jt_j$ have relatively prime
minimal polynomials.   Since the minimal polynomial of $t_i$
is the same as its characteristic polynomial, the same is true for
$c_it_i$, whence  the claim).

    Call a conjugacy class $C$ of $S_n$ bad if permutations $w$ in it do not satisfy
the condition $\gcd(a_1m_1,\cdots,a_rm_r,q-1)=1$. We want to upper bound
$|f-g|$ where $f,g$ are the proportion of regular semisimple elements in two
fixed cosets. Since each $T_w$ intersects each coset in
$\frac{|T_w|}{q-1}$ elements, and there are $|GL(n,q)|/|N(T_w)|$ maximal
tori of type $w$, it follows by the method of Section \ref{maximaltori}
(see for instance the proof of Theorem \ref{Gltorus}) that
\[ |f-g| \leq \frac{1}{|SL(n,q)|} \sum_{C \ \mathrm{bad}} \frac{|GL(n,q)||T_w|}{|N(T_w)|(q-1)}
\leq \frac{1}{|S_n|} \sum_{w \in C \atop {C \ \mathrm{bad}} } 1.\] The result now
follows from Theorem \ref{sntorierror}.
\end{proof}

    To conclude this subsection, we derive upper bounds (which
    will be used in Section \ref{mainresults}) for the proportion of
    regular semisimple elements in $GL(n,2)$ and $GL(n,3)$.

\begin{theorem} \label{glfiniteregss}
For $n>1$, the proportion of regular semisimple elements in $GL(n,2)$ is at most $5/6$.
\end{theorem}

\begin{proof}  The reasoning of Section
\ref{maximaltori} implies that the proportion of regular semisimple
elements in $GL(n,2)$ is at most the proportion of elements in $S_n$ with
at most 1 fixed point. From the cycle index of the symmetric groups, the
proportion of permutations with at most one fixed point is the coefficient
of $u^n$ in $\frac{1+u}{e^u(1-u)}$. This is easily seen to be at most
$5/6$.
\end{proof}

\subsection{SU}

    This subsection considers proportions of regular semisimple elements in the
unitary groups. Recall that $\tilde{N}(q;d)$ and $\tilde{M}(q;d)$ were
defined in Section \ref{preliminaries}.

\begin{theorem} \label{Uregss} \rm{(\cite{FNP})}
\begin{enumerate}
\item The fixed $q$, $n \rightarrow \infty$ proportion of regular semisimple elements
in $U(n,q)$ is $(1+1/q) \prod_{d \ odd}
(1-\frac{2}{q^d(q^d+1)})^{\tilde{N}(q;d)}$.
\item The fixed $q$, $n \rightarrow \infty$ proportion of regular semisimple elements
in $U(n,q)$ is at least $1-1/q-2/q^3+2/q^4$. For $q=2$ it is at least
$.414$, for $q=3$ it is at least $.628$, and for $q \geq 4$ it is at least
$.72$.
\end{enumerate}
\end{theorem}

    Using algebraic geometry, Guralnick and L\"{u}beck established the following result.

\begin{theorem} \rm{(\cite{GL})} For $n>2$, the proportion of regular semisimple
elements in $PSU(n,q)$ is at least $1-1/(q-1)-4/(q-1)^2$. In particular as
$q \rightarrow \infty$ this goes to 1 uniformly in $n$.
\end{theorem}

    The following result is due to Britnell.

\begin{theorem} \label{BritnellU} \rm{(\cite{B2}) } The fixed $q$, large $n$ limiting
proportion of regular semisimple elements in any coset of $H=SU(n,q)$ in
$U(n,q)$ is equal to the corresponding limit for $U(n,q)$. Furthermore the
same holds for a coset of any subgroup $H$ between $SU(n,q)$ and $U(n,q)$.
\end{theorem}

    Theorem \ref{differenceSU} is an analog of Theorem
    \ref{differenceSL} for the unitary groups.

\begin{theorem} \label{differenceSU} The difference in the proportion of regular
semisimple elements in any two cosets of $SU(n,q)$ in $U(n,q)$ is
at most $\frac{c_q \log(n)^3}{n^{1/2}}$, where $c_q$ is independent of
$n$.
\end{theorem}

\begin{proof} The proof is essentially the same as that of Theorem
\ref{differenceSL}. The group of possible determinants is the size $q+1$
subgroup of the multiplicative group of $\mathbb{F}_{q^2}$, so (using the
notation of Theorem \ref{differenceSL}), the condition for a conjugacy
class of $S_n$ to be bad is that $\gcd(a_1m_1,\cdots,a_rm_r,q+1) \neq 1$.
\end{proof}

The next result gives upper bounds for the proportion of regular semisimple
elements in unitary groups over small fields.

\begin{theorem} \label{Usmallregss}
\begin{enumerate}
\item For $n \geq 2$, the proportion of regular semisimple elements
in $U(n,2)$ is at most $.877$.

\item For $n \geq 2$, the proportion of regular semisimple elements in $U(n,3)$
is at most $.94$.
\end{enumerate}
\end{theorem}

\begin{proof} Using the cycle index of the unitary groups and Lemma \ref{Euleridentity},
one sees that the proportion of elements in $U(n,q)$ in which the
polynomial $(z-1)$ occurs with multiplicity 2 is the coefficient of $u^n$
in
\begin{eqnarray*}
& & \left[ \frac{u^2}{q^4(1+1/q)(1-1/q^2)} + \frac{u^2}{q^2(1+1/q)} \right] \frac{1}{1-u}
\prod_{i \geq 1} \left( 1+\frac{(-1)^i u}{q^{i}} \right)\\
& = & \frac{u^2q}{(q^2-1)(q+1)} \frac{1}{1-u} \prod_{i \geq 1} \left(
1+\frac{(-1)^i u}
{q^{i}} \right) \\
& = & \frac{u^2q}{(q^2-1)(q+1)} \frac{1}{1-u} \sum_{n=0}^{\infty}
\frac{(-1)^{{n+1 \choose 2}} u^n} {(q+1)(q^2-1) \cdots (q^n-(-1)^n)}.
\end{eqnarray*} Hence the coefficient of $u^n$ is at least
\begin{eqnarray*}
& & \frac{q}{(q^2-1)(q+1)} \sum_{r=0}^{n-2} \frac{(-1)^{{r+1 \choose 2}}}
{(q+1)(q^2-1) \cdots (q^r-(-1)^r)}\\
& \geq & \frac{q}{(q^2-1)(q+1)} \left(
1-\frac{1}{q+1}-\frac{1}{(q+1)(q^2-1)} \right).
\end{eqnarray*} Substituting $q=2$ and $q=3$ proves the result. \end{proof}

\subsection{Sp}

    This section considers the proportion of regular semisimple elements
in the symplectic groups $Sp(2n,q)$.

\begin{theorem} \label{Spregss} \rm{(\cite{FNP})}  Let $f= \gcd(q-1,2)$.
\begin{enumerate}
\item The fixed $q$, $n \rightarrow \infty$ proportion of regular semisimple elements
in $Sp(2n,q)$ is  \[ (1-1/q)^f \prod_{d \geq 1} \left(
1-\frac{2}{q^d(q^d+1)} \right)^{N^*(q;2d)}.\]

\item The proportion of part 1 is at least $.283$ for $q=2$, at least $.348$ for $q=3$,
at least $.453$ for $q=4$, at least $.654$ for $q=5$, at least $.745$ for $q=7$, at
least $.686$ for $q=8$, and at least $.797$ for $q \geq 9$.
\end{enumerate}
\end{theorem}

    As with the other groups, there is an important result due to Guralnick
and L\"{u}beck.

\begin{theorem} \label{uniformSp} \rm{(\cite{GL}) } For $n \geq 2$ the proportion of
regular semisimple elements in $Sp(2n,q)$ is at least
$1-2/(q-1)-1/(q-1)^2$. In particular, as $q \rightarrow \infty$, this goes
to 1 uniformly in $n$.
\end{theorem}

\begin{theorem} \label{boundedsmallSp} For $n \geq 1$, the proportion of regular semisimple
elements in $Sp(2n,q)$ is at most $.74$ for $q=4$, at most $.80$ for $q=5$, at most $.86$ for
$q=7$, and at most $.88$ for $q=8$.
\end{theorem}

\begin{proof} If an element of $Sp(2n,q)$ is regular semisimple, then its $z-1$ component is
trivial. From the cycle index of the symplectic groups, this occurs with probability equal
to the coefficient of $u^n$ in \[ \frac{\prod_{i \geq 1} (1-u/q^{2i-1})}{1-u} .\]

This in turn
is equal to \[ \sum_{r=0}^n [u^r] \prod_{i \geq 1}(1-u/q^{2i-1}), \]
where $[u^r] f(u)$ denotes the coefficient of $u^r$ in an expression $f(u)$.
By part 1
of Lemma \ref{Euleridentity}, this is at most \[ 1 - \frac{q}{q^2-1} + \frac{q^2}{(q^4-1)(q^2-1)},\]
which yields the theorem.
\end{proof}

\subsection{O}

    This subsection considers proportions of regular semisimple elements in the simple
groups $\Omega^{\pm}(n,q)$. Since $\Omega(2n+1,q)$ is isomorphic to
$Sp(2n,q)$ when $q$ is even, throughout this section we disregard this
case.

 We recall the  characterization of regular semisimple elements
in $\Omega$  (note that for the other classical groups, regular semisimple
is equivalent to having minimal polynomial equal to the characteristic
polynomial).

\begin{lemma} \label{characterize} Let $g \in \Omega^{\pm}(n,q)$ and
let $f(z)$ denote its characteristic polynomial.
Write $f(z)=f_0(z)(z-1)^a(z+1)^b$ (with $b=0$ if $q$ is even) with
$\gcd(f_0(z), z^2-1)=1$.   Note that $b$ is even and $a \equiv 2 \pmod n$.
Then $g$  is regular
semisimple if and only if $f_0$ is squarefree and $a, b \le 2$
(and if $a$ or $b$ is $2$, then $g$ is a scalar on that eigenspace).
\end{lemma}

 We now focus on the case of $q$ fixed. We begin with the case of $q$ even. Recall
that $\Omega^{\pm}(2n,q)$ is defined as the index 2 subgroup of
$O^{\pm}(2n,q)$ whose quasideterminant is equal to $1$. (Letting
$\dim(fix)$ denote the dimension of the fixed space of a matrix, the
quasideterminant of $\alpha$ is defined as $(-1)^{\dim(fix(\alpha))})$.

    The next result is clear.

\begin{lemma} \label{fixedparts} The dimension of the fixed space of a matrix
$\alpha$ is the number of parts in the partition corresponding to the
polynomial $z-1$ in the rational canonical form of $\alpha$. \end{lemma}

    From Lemma \ref{fixedparts} and the cycle index for the orthogonal groups
in \cite{Fu2}, it is clear that the cycle index of $\Omega^{\pm}$ is
obtained from the cycle index of $O^{\pm}$ simply by imposing the
additional restriction that the partition corresponding to the polynomial
$z-1$ has an even number of parts. We proceed to do this for the case of
regular semisimple elements.

    Let $rs_{\Omega^{\pm}}(2n,q)$ denote the proportion of regular semisimple
elements in $\Omega^{\pm}(2n,q)$. Let $RS_{\Omega^{\pm}}(u)$ be the
generating function defined by \[ RS_{\Omega^{\pm}}(u) = 1 + \sum_{n \geq
1} u^n \cdot rs_{\Omega^{\pm}}(2n,q).\]

    Theorem \ref{genfun} gives expressions for
$RS_{\Omega^{\pm}}$. For its statement we need some definitions. Letting
$rs_{Sp}(2n,q)$ denote the proportion of regular semisimple elements in
$Sp(2n,q)$, define the generating function $RS_{Sp}(u)$ by \[ RS_{Sp}(u) =
1+\sum_{n \geq 1} u^{n} \cdot rs_{Sp}(2n,q).\] This second generating
function was considered in \cite{FNP}. Also define, as in \cite{FNP}
\[ X_O(u) = \prod_{d \geq 1} (1-\frac{u^d}{q^d+1})^{N^*(q;2d)} \prod_{d
\geq 1} (1+\frac{u^d}{q^d-1})^{M^*(q;d)}.\]

\begin{theorem} \label{genfun} Suppose that $q$ is even.

\begin{enumerate}
\item \[ RS_{\Omega^+}(u) + RS_{\Omega^-}(u) = 2
 \left( 1+\frac{u}{2(q-1)}+\frac{u}{2(q+1)} \right) RS_{Sp}(u).\]

\item \[ 2 + RS_{\Omega^+}(u) - RS_{\Omega^-}(u) =
2 \left( 1+\frac{u}{2(q-1)}-\frac{u}{2(q+1)} \right) X_O(u) .\]

\end{enumerate}

\end{theorem}

\begin{proof} The proof runs along the lines of results of \cite{FNP} but
two additional points should be emphasized. First, the factors of two on
the right hand side come from the fact that $\Omega^{\pm}$ is an index 2
subgroup of $O^{\pm}$ when the characteristic is even. Second, Lemma
\ref{characterize} forces the partition coming from the $z-1$ component of
a regular semisimple element to be one of $(0),(1),(1,1)$. The choice
$(1)$ is ruled out since this partition must have even size since $2n$ is
even. Thus the partitions must be $(0)$ or $(1,1)$. In the second case one
must take care to consider $+,-$ types. The term $\frac{u}{2(q-1)}$ arises
from $+$ type and the term $\frac{u}{2(q+1)}$ arises from $-$ type.
\end{proof}

    Corollary \ref{regssO} calculates the fixed $q$, large $n$ limiting
proportion of regular semisimple elements in $\Omega^{\pm}(2n,q)$ in terms
of the corresponding proportions in $Sp(2n,q)$, which were discussed in
the previous subsection.

\begin{cor} \label{regssO} Suppose that $q$ is even. Then
\[ \lim_{n \rightarrow \infty} rs_{\Omega^+}(2n,q) =
\lim_{n \rightarrow \infty} rs_{\Omega^-}(2n,q) = \left( 1+\frac{q}{q^2-1}
\right) \lim_{n \rightarrow \infty} rs_{Sp}(2n,q).\] \end{cor}

\begin{proof} First we claim that
\[ \lim_{n \rightarrow \infty} (rs_{\Omega^+}(2n,q) +
rs_{\Omega^-}(2n,q)) = 2 \left( 1+\frac{q}{q^2-1} \right) \lim_{n
\rightarrow \infty} rs_{Sp}(2n,q).\] This follows from Lemma
\ref{Taylorcoeff} and from the facts that
$(1+\frac{u}{2(q-1)}+\frac{u}{2(q+1)})$ is analytic in a circle of radius
greater than $1$ and that $(1-u) RS_{Sp}(u)$ is analytic in a circle of
radius greater than $1$ (\cite{FNP}). Next we claim that \[ \lim_{n
\rightarrow \infty} (rs_{\Omega^+(2n,q)} - rs_{\Omega^-(2n,q)}) = 0.\]
This follows from the result in \cite{FNP} that the $n \rightarrow \infty$
limit of the coefficient of $u^n$ in $X_0(u)$ is $0$. \end{proof}

    Next we consider the case of $q$ odd. First, we derive the large $n$
limiting proportion of regular semisimple elements in $SO^{\pm}$. Then it
is proved that these proportions are equal to the corresponding
proportions for $\Omega^{\pm}$.

\begin{theorem} \label{largeregssoddo} Suppose that $q$ is odd.
Let $r(q)$ denote
the $n \rightarrow \infty$ proportion of regular semisimple elements in
$SO(2n+1,q)$.  Then  $r(q)=(1+\frac{q}{q^2-1})$ multiplied by the corresponding
limiting proportion in $Sp(2n,q)$. In particular,  $r(q)$ is at
least $.478$ for $q=3$ and at least $.790$ for $q \geq 5$.
\end{theorem}

\begin{proof} An element of $SO(2n+1,q)$ is regular semisimple if and only
if all polynomials other than $z \pm 1$ occur with multiplicity at most 1,
the polynomial $z-1$ occurs with multiplicity exactly 1, and the
polynomial $z+1$ occurs with multiplicity 0 or 2 (and in the latter case
the $z+1$ piece of the underlying vector space decomposes as the direct
sum of two invariant 1 dimensional subspaces). Letting $rs_{SO}(2n+1,q)$
denote the proportion of regular semisimple elements in $SO(2n+1,q)$, it
follows from the reasoning of \cite{FNP} and the fact that $SO^+(2n+1,q)$
is isomorphic to $SO^-(2n+1,q)$ that
\begin{eqnarray*}
& & 1+\sum_{n \geq 1} u^n \cdot rs_{SO}(2n+1,q)\\
& = & 1+ \sum_{n \geq 1} \frac{u^n}{2} \left( rs_{SO^+}(2n+1,q)
+ rs_{SO^-}(2n+1,q)\right)\\
& = & \left( 1+\frac{u}{2(q-1)}+\frac{u}{2(q+1)} \right) (1/2 + 1/2) \cdot
RS_{Sp}(u).
\end{eqnarray*} The term $(1+\frac{u}{2(q-1)}+\frac{u}{2(q+1)})$
corresponds to the $z+1$ piece of the characteristic polynomial of the
element, the term $(1/2+1/2)$ corresponds to the $z-1$ piece, and the term
$RS_{Sp}(u)$ arises from the other factors of the characteristic
polynomial. The theorem now follows from Lemma \ref{Taylorcoeff} since
$(1-u) RS_{Sp}(u)$ is analytic in a circle of radius greater than 1.

The last statement now follows by   Theorem \ref{Spregss}.
\end{proof}

As in Section \ref{maximaltori}, we say that an element of an
$n$-dimensional orthogonal group is strongly regular semisimple if it has $n$ distinct
eigenvalues. Arguing as in Theorem \ref{largeregssoddo} proves the following
result.

\begin{theorem} \label{addone} Suppose $q$ is odd.  Let $s(q)$ be the  $n \rightarrow \infty$
proportion of strongly regular semisimple elements in $SO(2n+1,q)$.  Then $s(q)$ is equal to
the $n \rightarrow \infty$ proportion of regular semisimple elements in $Sp(2n,q)$.
By Theorem \ref{Spregss}, $s(q)$ is at least  $.348$ for $q=3$ and at least $.654$
for $q \geq 5$.
\end{theorem}

For even dimensional orthogonal groups, one has the following results.

\begin{theorem} \label{largeregsseveno} Suppose that $q$ is odd.  Let $r(q)$
denote the $n \rightarrow \infty$ proportion of regular semisimple elements in
$SO^{\pm}(2n,q)$.  Then $r(q)=(1+\frac{q}{q^2-1})^2$ multiplied by the
corresponding limiting proportion in $Sp(2n,q)$. By Theorem \ref{Spregss},
it follows that $r(q)$ is at least  $.657$ for $q=3$ and at least $.954$ for $q \geq 5$.
\end{theorem}

\begin{proof} An element of $SO^{\pm}(2n,q)$ is regular semisimple if and only
if all polynomials other than $z \pm 1$ occur with multiplicity at most 1,
and the polynomials $z \pm 1$ occur with multiplicity 0 or 2 (in the
multiplicity 2 case the piece of the vector space with characteristic
polynomial $z \pm 1$ is a sum of 1 dimensional invariant spaces). Let
$rs_{SO^{\pm}}(2n,q)$ denote the proportion of regular semisimple elements
in $SO^{\pm}(2n,q)$. The reasoning of \cite{FNP} implies that
\begin{eqnarray*}
& & 1 + \sum_{n \geq 1} u^n \left( \frac{rs_{SO^+}(2n,q)}{2} +
  \frac{rs_{SO^-}(2n,q)}{2} \right)\\
& = & \left( 1+\frac{u}{2(q-1)}+\frac{u}{2(q+1)} \right)^2 \cdot
RS_{Sp}(u).
\end{eqnarray*} The result now follows from Lemma \ref{Taylorcoeff}
since $(1-u) RS_{Sp}(u)$ is analytic in a circle of radius greater than 1,
and \[ \lim_{n \rightarrow \infty} rs_{SO^+}(2n,q) = \lim_{n \rightarrow
\infty} rs_{SO^-}(2n,q)\] by a generating function argument similar to
that of Corollary \ref{regssO}. \end{proof}

\begin{theorem} \label{addtwo} Suppose $q$ is odd.  Let $s(q)$ be the  $n \rightarrow \infty$
proportion of strongly regular semisimple elements in $SO^{\pm}(2n,q)$.
Then $s(q)$  is equal to
the $n \rightarrow \infty$ proportion of regular semisimple elements in $Sp(2n,q)$.
By Theorem \ref{Spregss}, $s(q)$ is at least $.348$ for $q=3$ and at least $.654$
for $q \geq 5$.
\end{theorem}

   Lemmas \ref{permforO} and \ref{involution} will be useful in passing from
$SO^{\pm}$ to $\Omega^{\pm}$ in odd characteristic.

\begin{lemma} \label{permforO}    Let $e(n)$ be the proportion of elements in
$B_n$ that have the property  that all even length negative cycles occur with even
multiplicity.  Then $\lim_{n \rightarrow \infty} e(n)=0$.
\end{lemma}

\begin{proof} Recall the notation $f<<g$ from Subsection \ref{asymptot}. By
the cycle index of the hyperoctahedral groups, the proportion
of elements of $B_n$ in which all even length negative cycles have even
multiplicity is the coefficient of $u^n$ in
\begin{eqnarray*}
& & \prod_{i \ odd} e^{\frac{u^i}{i}} \prod_{i \ even}
e^{\frac{u^{i}}{2i}}
 \left( \sum_{j \geq 0, \ even} \frac{u^{ij}}{(2i)^j j!} \right)\\
& << & \prod_{i \ odd} e^{\frac{u^i}{i}} \prod_{i \ even}
e^{\frac{u^{i}}{2i}}
 \left( \sum_{j \geq 0, \ even} \frac{u^{ij}}{(2i)^j 2^{j/2} (j/2)!} \right) \\
& = & \prod_{i \ odd} e^{\frac{u^i}{i}} \prod_{i \ even}
 e^{\frac{u^{i}}{2i}+\frac{u^{2i}}{8i^2}}\\
& = & \frac{1}{(1-u)} \prod_{i \ even}
e^{\frac{u^{2i}}{8i^2}-\frac{u^i}{2i}}.
\end{eqnarray*} It follows from Lemma \ref{Taylorcoeff} that as $n \rightarrow
\infty$, the coefficient of $u^n$ in this expression goes to 0.
\end{proof}

\begin{lemma} \label{involution} Let $q$ be odd.    Then the spinor norm
$\theta(t)$ of an involution
$t$ in $SO^{\pm}(n,q)$ depends only on its $-1$ eigenspace which has dimension
2d and is of type $\epsilon$.  More precisely,  $\theta(t)$ is trivial if and only if
  $\epsilon = +$ and $d$ is even or $d$ is odd and $\epsilon$ is a square modulo
  $q$.
  \end{lemma}

  \begin{proof}   Write $V = V_1 \perp V_2$ where $V_1$ is the fixed space
  of $t$.   Clearly, $\theta(t) = \theta(t|_{V_2})$. Write
  $V_2$ as an orthogonal sum of $d$ copies of two dimensional subspaces.
  Note that if two of the summands are of the same type, then $t$ restricted
  to the sum of those two summands has spinor norm $1$ (since it will be
  a square).  So we are reduced to considering the $2$ and $4$ dimensional
  cases (corresponding to $d$ odd and $d$ even).    If $d$ is odd,  then
  the spinor norm is  trivial if and only if $4|(q - \epsilon)$.  While if
  $d$ is even, the spinor norm will be trivial if and only if each of the two
  summands has the same type, whence the result.
  \end{proof}


Now we can prove the following result.

\begin{theorem} \label{Omegaregssequal} Let $q$ be odd. \begin{enumerate}
\item The large $n$ limiting proportion of regular semisimple elements in
$\Omega(2n+1,q)$ is equal to the corresponding limiting proportion in
$SO(2n+1,q)$, and the difference between these proportions for a given $n$
is bounded independently of $q$.

\item The large $n$ limiting proportion of regular semisimple elements in
$\Omega^{\pm}(2n,q)$ is equal to the limiting proportion in
$SO^{\pm}(2n,q)$, and the difference between these proportions for a given
$n$ is bounded independently of $q$.

\item The large $n$ limiting proportion of strongly regular semisimple elements in
$\Omega(2n+1,q)$ is equal to the corresponding limiting proportion in
$SO(2n+1,q)$, and the difference between these proportions for a given $n$
is bounded independently of $q$.

\item The large $n$ limiting proportion of strongly regular semisimple elements in
$\Omega^{\pm}(2n,q)$ is equal to the limiting proportion in
$SO^{\pm}(2n,q)$, and the difference between these proportions for a given
$n$ is bounded independently of $q$.

\end{enumerate}
\end{theorem}

\begin{proof}  For the first
part of the theorem, the group in question is $SO(2n+1,q)$ and has Weyl
group the hyperoctahedral group $B_n$; each $w$ corresponds to a product of
signed cycles. Let $T_w$ be a maximal torus in G  and decompose the space
according to $w$; i.e. it is an orthogonal sum
$V_0 \perp V_1 \perp \ldots \perp V_r$ where $V_0$ is a $1$-space where $T_w$
is trivial  and each $V_i$ has
even dimension $2d_i$ where $T_w$ is irreducible on the
$2d_i$ space (which is therefore of - type) or the space is of + type and
the space splits as a direct sum of 2 subspaces of dimension $d_i$ and
$T_w$ acts dually on them.

For a given cycle length and type, let $U$ be the direct sum of all the
$V_i$ with  that given cycle length and type.   We call $U$ a homogeneous
piece.  So $\dim U = 2d$ where $d =d_im$ for some $m \ge 1$.

  We seek  some homogeneous piece $U$
  and an element $s$ in $T_w$  trivial on $U^{\perp}$
such that $s$ has nontrivial spinor norm and if $t$ is in $T_w$, then $st$ is
regular semisimple if and only if $t$ is (for then this is a bijection
between regular semisimple elements in $T_w$ with trivial and nontrivial spinor norm).

Note that the the characteristic
polynomials of regular semisimple elements on different homogeneous pieces
are relatively prime and thus it suffices to check that
$st$ is regular semisimple on $U$ if and only if $t$ is.

    By Lemma \ref{involution}, such an element exists provided that $T_w$
corresponds to a conjugacy class of signed permutations $w$ such that at
least one of the following holds:

\begin{enumerate}
\item There is an even length negative cycle which occurs with odd multiplicity.
\item There is an odd cycle length $>1$ with odd multiplicity and negative type and
 $q=1$ mod 4.
\item There is an odd cycle length $>1$ with odd multiplicity and positive type and
 $q=3$ mod 4.
\end{enumerate}

    Call $w$ good it if satisfies any of these properties. From Lemma \ref{permforO},
    it follows that almost all elements
    in the Weyl group are good (and in
    particular satisfy the first property), so by the method of Theorem
    \ref{differenceSL}, the difference in the proportion of regular
    semisimple elements in the two cosets of $\Omega(2n+1,q)$ is at most the
    proportion of elements $w$ not satisfying the above
    properties, so this goes to 0 as $n \rightarrow \infty$, uniformly in $q$.

    For the second part of the theorem one must work in $D_n$ rather than $B_n$,
but any of the 3 conditions in the first part still ensure that the
conjugacy class is good so the result follows by a minor modification of
Lemma \ref{permforO}.

Parts (3) and (4) follow by precisely the same argument (indeed the estimate can only
get better since we only need consider maximal tori which contain strongly regular
semisimple elements and the construction above still takes strongly regular semisimple elements
to strongly regular semisimple elements). \end{proof}

\section{Nearly regular semisimple elements} \label{nearlyregss}

    This section proves that with high probability, an element of a finite classical
group is regular semisimple on a space of small codimension. More
precisely, the following result is established.

\begin{theorem} \label{nearly} Let $G$ be one of $GL(n,q)$, $U(n,q)$, $Sp(2n,q)$,
or $O^{\pm}(n,q)$. Then there are universal constants $c_1,c_2$ such that
for any $r>0$, the probability that an element of $G$ is regular
semisimple on some subspace of codimension $\leq c_1+r$ is at least
$1-c_2/r^2$.
\end{theorem}

    As the proof of Theorem \ref{nearly} will show, values of the constants
$c_1,c_2$ can be worked out though it is tedious and not necessary for the
present paper. For example we prove that when $G=GL(n,q)$, one can take
$c_1=\frac{2}{q(1-1/q)^3(1-1/q^{1/2})}$, which is at most $28$ since $q
\geq 2$.

\begin{proof} We give full details only for $GL(n,q)$, but indicate what changes
are needed for the other groups in the statement of the theorem.

For $G=GL,U$ or $Sp$, let $D(\alpha)$ be the sum of the degrees (counting multiplicity)
of the irreducible factors of the characteristic polynomial of $\alpha \in
G$ which occur with multiplicity greater than one. For $G=O$,  let
$D(\alpha)$ be the sum of the degrees (counting multiplicity) of the
irreducible factors of the characteristic polynomial of $\alpha \in G$
which occur with multiplicity greater than one and of the irreducible
factors corresponding to $z \pm 1$; the only reason for the different
definition in the orthogonal case is to simplify the generating function.
Our strategy is to obtain upper bounds for  the expected value and variance of
$D(\alpha)$ and to then apply Chebyshev's inequality, which states that
for any random variable $X$ with mean $\mu$ and variance $\sigma^2$, the
probability that $|X-\mu| \geq a$ is at most $\frac{\sigma^2}{a^2}$.

    Using the partition notation in Subsection \ref{prelimGLcycle}, one sees that
the generating function for the random variable $D$ on $GL(n,q)$ in the
variable $t$ is the coefficient of $u^n$ in \[ \prod_{d \geq 1}
\left( \sum_{\lambda} \frac{(ut)^{d |\lambda|}}{q^{d \sum (\lambda_i')^2}
\prod_i (1/q^d)_{m_i(\lambda)}} + \frac{u^d}{q^d-1} - \frac{u^dt^d}{q^d-1}
\right)^{N(q;d)} .\] By Lemma \ref{Stongs} this is equal to \[ F(u,t) :=
\prod_{d \geq 1} \left(\prod_{i \geq 1} (1-\frac{u^d t^d}{q^{id}})^{-1} +
\frac{u^d}{q^d-1} - \frac{u^dt^d}{q^d-1} \right)^{N(q;d)}.\] To compute
the expected value of $D$, one differentiates with respect to $t$ and then
sets $t=1$. Doing this gives the coefficient of $u^n$ in
\begin{eqnarray*}
& & \sum_{d \geq 1} N(q;d) \prod_{i \geq 1}
(1-\frac{u^d}{q^{id}})^{-N(q;d)+1}
\prod_{k \neq d} \prod_{i \geq 1} (1-\frac{u^k}{q^{ik}})^{-N(q;k)}\\
& & \cdot d/dt \left(\prod_{i \geq 1} (1-\frac{u^d t^d}{q^{id}})^{-1} +
\frac{u^d}{q^d-1} - \frac{u^dt^d}{q^d-1}\right)_{t=1}.
\end{eqnarray*} It is straightforward to see that
\begin{eqnarray*}
& & d/dt \left(\prod_{i \geq 1} (1-\frac{u^d t^d}{q^{id}})^{-1} + \frac{u^d}{q^d-1} - \frac{u^dt^d}{q^d-1}\right)_{t=1}\\
& = & \left( \prod_{i \geq 1} (1-\frac{u^d}{q^{id}})^{-1} \sum_{j \geq 1}
(1-\frac{u^d}{q^{jd}})^{-1} \frac{du^d}{q^{jd}} \right) -
\frac{du^d}{q^d-1}.
\end{eqnarray*}

 Thus the expected value of $D$ is the coefficient of $u^n$ in
\begin{eqnarray*}
& & \sum_{d \geq 1} N(q;d) \prod_{k \geq 1} \prod_{i \geq 1}
(1-\frac{u^k}{q^{ik}})^{-N(q;k)} \\
& & \cdot \left[ \left(\sum_{j \geq 1} (1-\frac{u^d}{q^{jd}})^{-1}
\frac{du^d}{q^{jd}} \right) - \frac{du^d}{q^d-1} \prod_{i \geq 1}
(1-\frac{u^d}{q^{id}}) \right].
\end{eqnarray*} Using part 1 of Lemma \ref{set1incycle} and Lemma \ref{Euleridentity}, this
simplifies to the coefficient of $u^n$ in
\begin{eqnarray*}
& & \frac{1}{1-u} \sum_{d \geq 1} N(q;d) \left[ \left(\sum_{j \geq 1}
(1-\frac{u^d}{q^{jd}})^{-1} \frac{du^d}{q^{jd}} \right) - \prod_{i \geq 1}
(1-\frac{u^d}{q^{id}}) \cdot
\frac{du^d}{q^d-1} \right]\\
& = & \sum_{d \geq 1} \frac{d N(q;d)}{1-u} \left[\sum_{j \geq 1} \sum_{r \geq 1} \frac{u^{dr}}{q^{jdr}} + \sum_{r \geq 1} \frac{(-1)^r}{q^d-1} \frac{u^{dr}}{(q^{d(r-1)}-1) \cdots (q^d-1)} \right]\\
& = & \sum_{d \geq 1} \frac{d N(q;d)}{1-u} \left[\sum_{r \geq 1} \frac{u^{dr}}{q^{dr}-1} + \sum_{r \geq 1} \frac{(-1)^r}{q^d-1} \frac{u^{dr}}{(q^{d(r-1)}-1) \cdots (q^d-1)} \right]\\
& = & \frac{1}{1-u} \sum_{d \geq 1} d N(q;d) \sum_{r \geq 2} \left(
\frac{u^{dr}}{q^{dr}-1} + \frac{(-1)^r}{q^d-1} \frac{u^{dr}}{(q^{d(r-1)}-1) \cdots (q^d-1)} \right).
\end{eqnarray*} Note that the $r=1$ term has canceled, which is crucial. Using
the bound $d N(q;d) \leq q^d$, and the notation $f << g$ from Section \ref{asymptot},
one sees that the mean of $D$ is at most
the coefficient of $u^n$ in
\begin{eqnarray*}
& & \frac{1}{1-u} \sum_{d \geq 1} q^d \sum_{r \geq 2} \left( \frac{u^{dr}}{q^{dr}-1} +
\frac{1}{(q^d-1)} \frac{u^{dr}}{(q^{d(r-1)}-1) \cdots (q^d-1)} \right)\\
& << & \frac{1}{1-u} \sum_{d \geq 1} q^d \sum_{r \geq 2} \frac{u^{dr}}{q^{dr}}
\left( \frac{1}{1-1/q^{dr}}+\frac{1}{(1-1/q^d)(1-1/q^{d(r-1)})} \right)\\
& << & 2 \left( \frac{1}{1-1/q} \right)^2 \frac{1}{1-u} \sum_{d \geq 1} q^d \sum_{r \geq 2} \frac{u^{dr}}{q^{dr}}\\
& = & 2 \left( \frac{1}{1-1/q} \right)^2 \frac{1}{1-u} \sum_{m \geq 2} \frac{u^m}{q^m} \sum_{r|m \atop r \geq 2} q^{m/r}\\
& << & 2 \left( \frac{1}{1-1/q} \right)^3 \frac{1}{1-u} \sum_{m \geq 2}
\frac{u^m}{q^{m/2}}. \end{eqnarray*} This is at most \[
\frac{2}{q(1-1/q)^3(1-1/q^{1/2})} \leq 28\] for $q \geq 2$.

    To finish the proof for $GL(n,q)$, we sketch an argument that $\sigma$ (the
variance of $D$) is finite, and bounded independently of $n$ and $q$. It
is convenient to define \[ S(u,t,d) = \prod_{i \geq 1} (1-\frac{u^d
t^d}{q^{id}})^{-1} + \frac{u^d}{q^d-1} - \frac{u^dt^d}{q^d-1}.\] Observe
that the expected value of $D(D-1)$ is
\[ \frac{d}{dt} \frac{d}{dt} F(u,t)_{t=1} = \frac{d}{dt} \left[ F(u,t) \sum_{d
\geq 1} \frac{N(q;d)}{S(u,t,d)} \frac{d}{dt} S(u,t,d) \right]_{t=1}.\] We
know from Lemma \ref{set1incycle} that $F(u,1)=\frac{1}{1-u}$ and that the
coefficient of $u^n$ in $\frac{d}{dt}F(u,t)_{t=1}$ is bounded by a
constant independent of $n,q$ (this was the computation of the mean of
$D$). Combining this with an analysis of the first and second derivatives
of $S(u,t,d)$ at $t=1$ (using part 2 of Lemma \ref{Euleridentity}) proves
the result. The essential point (as in the computation of the mean of $D$)
is that the coefficient of $t^d$ in $S(u,t,d)$ vanishes.

    For the case of the unitary groups, one sees that the generating function for
the random variable $D$ on $U(n,q)$ in the variable $t$ is the coefficient
of $u^n$ in
\begin{eqnarray*}
& & \prod_{d \geq 1 } \left( \prod_{i \geq 1} \left(1+\frac{(-1)^iu^dt^d}
{q^{id}}\right)^{-1} + \frac{u^d}{q^d+1} - \frac{u^dt^d}{q^d+1} \right)^{\tilde{N}(q;d)}\\
& \cdot & \prod_{d \geq 1} \left( \prod_{i \geq 1}
\left(1-\frac{u^{2d}t^{2d}}{q^{2id}}\right)^{-1} +
\frac{u^{2d}}{q^{2d}-1}-\frac{u^{2d}t^{2d}}{q^{2d}-1}
\right)^{\tilde{M}(q;d)}.
\end{eqnarray*}

    For the case of the symplectic groups, one sees that the generating function
for the random variable $D$ on $Sp(2n,q)$ in the variable $t$ is the
coefficient of $u^n$ in
\begin{eqnarray*}
& &  \prod_{d \geq 1}
\left( \prod_{i \geq 1} (1+\frac{(-1)^i u^d t^d}{q^{id}})^{-1} + \frac{u^d}{q^d+1}-\frac{u^dt^d}{q^d+1} \right)^{N^*(q;2d)}\\
& & \cdot \prod_{d \geq 1} \left( \prod_{i \geq 1}
(1-\frac{u^dt^d}{q^{id}})^{-1}+\frac{u^d}{q^d-1}-\frac{u^dt^d}{q^d-1}
\right)^{M^*(q;d)} \cdot \prod_{r=1}^{\infty} (1-\frac{ut}{q^{2r-1}})^{-f}
\end{eqnarray*} where $f=\gcd(q-1,2)$.

To write down the generating functions for the random variable $D$ on the orthogonal groups, we first consider the case of even dimensional
orthogonal groups in even characteristic (note that odd dimensional orthogonal groups in even characteristic are isomorphic to
symplectic groups). The sum of the generating functions for the random variable $D$ on $O^+(2n,q)$ and $O^-(2n,q)$ is the
coefficient of $u^n$ in
\begin{eqnarray*}
& &  \prod_{d \geq 1}
\left( \prod_{i \geq 1} (1+\frac{(-1)^i u^d t^d}{q^{id}})^{-1} + \frac{u^d}{q^d+1}-\frac{u^dt^d}{q^d+1} \right)^{N^*(q;2d)}\\
& & \cdot \prod_{d \geq 1} \left( \prod_{i \geq 1}
(1-\frac{u^dt^d}{q^{id}})^{-1}+\frac{u^d}{q^d-1}-\frac{u^dt^d}{q^d-1}
\right)^{M^*(q;d)}\\
& & (1+ut) \cdot \prod_{r=1}^{\infty} (1-\frac{ut}{q^{2r-1}})^{-1}
\end{eqnarray*} The difference of the generating functions for the random variable $D$ on $O^+(2n,q)$ and $O^-(2n,q)$ is the
coefficient of $u^n$ in
\begin{eqnarray*}
& &  \prod_{d \geq 1}
\left( \prod_{i \geq 1} (1 - \frac{(-1)^i u^d t^d}{q^{id}})^{-1} - \frac{u^d}{q^d+1} + \frac{u^dt^d}{q^d+1} \right)^{N^*(q;2d)}\\
& & \cdot \prod_{d \geq 1} \left( \prod_{i \geq 1}
(1-\frac{u^dt^d}{q^{id}})^{-1}+\frac{u^d}{q^d-1}-\frac{u^dt^d}{q^d-1}
\right)^{M^*(q;d)}\\
& & \cdot \prod_{r=1}^{\infty} (1-\frac{ut}{q^{2r}})^{-1}.
\end{eqnarray*} Knowing the sum and difference of the generating functions of the random
variable $D$ on $O^+(2n,q)$ and $O^-(2n,q)$ allows one
to solve for the generating functions of $D$ on each of $O^+(2n,q)$ and $O^-(2n,q)$.

Next we treat the case of odd characteristic orthogonal groups. The sum of the generating functions for the random variable
$D$ on $O^+(n,q)$ and $O^-(n,q)$ is the coefficient of $u^n$ in
\begin{eqnarray*}
& &  \prod_{d \geq 1}
\left( \prod_{i \geq 1} (1+\frac{(-1)^i u^{2d} t^{2d}}{q^{id}})^{-1} + \frac{u^{2d}}{q^d+1}-\frac{u^{2d}t^{2d}}{q^d+1} \right)^{N^*(q;2d)}\\
& & \cdot \prod_{d \geq 1} \left( \prod_{i \geq 1}
(1-\frac{u^{2d} t^{2d}}{q^{id}})^{-1}+\frac{u^{2d}}{q^d-1}-\frac{u^{2d}t^{2d}}{q^d-1}
\right)^{M^*(q;d)}\\
& & (1+ut)^2 \cdot \prod_{r=1}^{\infty} (1-\frac{u^2t^2}{q^{2r-1}})^{-2}
\end{eqnarray*} The difference of the generating functions for the random variable $D$ on $O^+(n,q)$ and $O^-(n,q)$ is the
coefficient of $u^n$ in
\begin{eqnarray*}
& &  \prod_{d \geq 1}
\left( \prod_{i \geq 1} (1 - \frac{(-1)^i u^{2d} t^{2d}}{q^{id}})^{-1} - \frac{u^{2d}}{q^d+1} + \frac{u^{2d}t^{2d}}{q^d+1} \right)^{N^*(q;2d)}\\
& & \cdot \prod_{d \geq 1} \left( \prod_{i \geq 1}
(1-\frac{u^{2d} t^{2d}}{q^{id}})^{-1}+\frac{u^{2d}}{q^d-1}-\frac{u^{2d}t^{2d}}{q^d-1}
\right)^{M^*(q;d)}\\
& & \cdot \prod_{r=1}^{\infty} (1-\frac{u^2t^2}{q^{2r}})^{-2}.
\end{eqnarray*} As in the case of even characteristic, knowing the sum and difference of the generating functions of $D$ on $O^+(n,q)$ and $O^-(n,q)$ allows one
to solve for the generating functions of $D$ on each of $O^+(n,q)$ and $O^-(n,q)$.

    What makes all of the above generating functions tractable is that they involve many products.
To compute the mean of $D$ it is feasible to use the product rule to
differentiate it with respect to $t$ and then set $t=1$. To carry out the
program as for $GL$, one uses Lemma \ref{set1incycle},
\ref{Euleridentity}, and the expressions for $\tilde{N}(q;d)$, $\tilde{M}(q;d)$,
$N^*(q;d)$, and $M^*(q;d)$ in Subsection \ref{polyenum}. The computation of the
variance of $D$ runs along the same lines. \end{proof}

\section{Main results: proportion of derangements in subspace actions}
\label{mainresults}

    This section proves the main results of this paper; these can be subdivided
into two types of results. The first set establishes a strengthening of the Boston--Shalev
conjecture in the case of subspace actions: we show that for a primitive
subspace action of a simple classical group $G$ with $|G|$ sufficiently
large, the proportion of elements which are both semisimple regular
and derangements is at least $\delta \geq .016$ (and
often much better). The second set of results shows that when the
dimension and codimension of the subspace grow to infinity, the proportion
of derangements goes to 1. Moreover, in both cases we give some results
for proportions of derangements in cosets of simple finite classical
groups $H$ in groups $G$ with $G/H$ cyclic.   By the results of
Section \ref{Large Fields}, it suffices to take $q$ fixed and we do so
for the rest of the section.

\subsection{SL} \label{Glmainresults}

    Recall that we are dealing with asymptotic results: thus the order of the
group goes to infinity. This subsection considers the action of $GL(n,q)$
and cosets of $SL(n,q)$ in $GL(n,q)$ on $k$ dimensional subspaces.
Throughout we suppose that $1 \leq k \leq n/2$, as the actions of elements on $k$
spaces and on $n-k$ spaces have the same number of fixed points, and so
in particular the proportion of derangements is the same in both cases.

First, we show that for any fixed $k$ ($1 \leq k \leq n/2$), the
proportion of elements which are regular semisimple and derangements on $k$-spaces
is uniformly bounded away from 0. Theorem \ref{eigenequal} reduces to the case that $G=GL(n,q)$.

\begin{theorem} \label{eigenequal} Let $gSL(n,q)$ be a coset of $SL(n,q)$ in $GL(n,q)$.
For $k$ fixed, $q$ fixed, and $n \rightarrow \infty$, the proportion of
elements of $gSL(n,q)$ which are regular semisimple and derangements on $k$-spaces is equal to the
proportion of elements of $GL(n,q)$ which are regular semisimple and derangements on $k$-spaces.
\end{theorem}

\begin{proof} We argue as in the proof of Theorem \ref{differenceSL}. The only difference is
that instead of summing over ``bad'' conjugacy classes $C$, we sum over bad conjugacy classes
with the additional property that $w \in C$ fixes a $k$-set.
\end{proof}

    Lemma \ref{limiting} calculates the $n \rightarrow \infty$ limiting proportion of elements of
$GL(n,q)$ which are eigenvalue free and regular semisimple. We remark that the proportion
of eigenvalue free elements was first calculated by Stong, and later studied in \cite{NP}.

\begin{lemma} \label{limiting}
The fixed $q$, $n \rightarrow \infty$ limit of the proportion of elements of
$GL(n,q)$ which are eigenvalue free and regular semisimple is equal to
\[ \frac{1-1/q}{\left( 1 + \frac{1}{q-1} \right )^{q-1}} .\] This is easily
seen to be at least $1/4$.
\end{lemma}

\begin{proof} The papers \cite{Fu2}, \cite{W2} use generating functions to show that the
fixed $q$, large $n$ limiting proportion of regular semisimple elements of $GL(n,q)$ is $1-1/q$.
A very minor modification of their arguments (removing the degree $1$ term in the generating
function of regular semisimple probabilities) proves the result.
\end{proof}

\begin{theorem} \label{correctSL} Suppose that $1 \leq k \leq n/2$ is fixed.
For $|SL(n,q)|$ sufficiently large, the proportion of elements in any
coset $gSL(n,q)$ of $GL(n,q)$ which are regular semisimple and derangements on $k$-spaces is at
least $1/16$.
\end{theorem}

\begin{proof}
Recall that we are taking fixed $q$, as large $q$ was handled in
Section \ref{Large Fields}. By Theorem \ref{eigenequal},
it is sufficient to work with $GL(n,q)$. Theorems \ref{Gltorus} and
\ref{Dix} imply that the proportion of elements of $GL(n,q)$ which are
regular semisimple and fix a $k$-space is at most $2/3$. Hence Theorem
\ref{Glregss} gives that for $q \geq 4$, the large $n$ proportion of elements of
$GL(n,q)$ which are regular semisimple and derangements on $k$-spaces is
at least $3/4-2/3 \geq .08$. For $q=3,k=1$ the result follows from Lemma
\ref{limiting}. For $q=3, k \geq 2$, Theorem
\ref{Gltorus} and Lemma \ref{atmost2fixed} imply that for $n$ sufficiently
large the proportion of elements in $GL(n,q)$ which are regular semisimple
and derangements on $k$-spaces is at least $1/16$ since
$1/16<(2/3)-(3/5)=1/15$.

    Finally we consider the case $q=2$. When $k=1$, the result follows from Lemma
\ref{limiting}. Suppose that $k \geq 2$. Let $H \subset
GL(n,2)$ be a stabilizer of a $k$-space. The proportion of regular
semisimple elements in $H$ is at most the product of the proportions of
regular semisimple elements in $GL(k,2)$ and $GL(n-k,2)$. The former
proportion is at most $5/6$ by Theorem \ref{glfiniteregss} and the
latter proportion goes to $1/2$ as $n \rightarrow \infty$ by Theorem
\ref{Glregss}. It is easily seen that the proportion of elements of $GL(n,2)$
which are regular semisimple and fix a $k$-space is at most the proportion of
elements of $H$ which are regular semisimple. (Indeed, the proportion of elements
of $GL(n,2)$ which are regular semisimple and fix a $k$-space is at most the number of conjugates of $H$
multiplied by the number of regular semisimple elements of $H$, and then divided
by $|GL(n,2)|$. Since $H$ is maximal in $GL(n,2)$ and not normal, the number
of conjugates of $H$ is equal to $|GL(n,2)|/|H|$, which proves the claim).
Since the $n \rightarrow \infty$ limiting proportion of regular semisimple
elements in $GL(n,2)$ is $1/2$, the theorem follows as $(1/2)-(1/2)(5/6)=1/12$.
\end{proof}

    Next we treat the case that $k \rightarrow \infty$. Note that for $q
\rightarrow \infty$, we already have very good estimates on the proportion
of derangements (see Theorem \ref{thm: large q}).

\begin{theorem} \label{GLlargek0} Suppose that $1 \leq k \leq n/2$ with $q$
fixed.
 If $k \rightarrow \infty$, the proportion of elements of $GL(n,q)$ which
 are derangements on $k$-spaces $\rightarrow 1$. More precisely, there are
 universal constants $A,B$ such that for any $\epsilon > 0$ and $k$, the proportion of
derangements of $GL(n,q)$ on $k$-spaces is at least  \[1- \epsilon -
\frac{A}{\epsilon (k - B/\sqrt{\epsilon})^{.01} }. \]
\end{theorem}

\begin{proof} Recall from Subsection
\ref{prelimGLcycle} that the conjugacy classes of $GL(n,q)$ are
parameterized by associating to each monic irreducible polynomial $\phi$
over $\mathbb{F}_q$ (disregarding the polynomial $\phi=z$) a partition
$\lambda_{\phi}$ such that $\sum \deg(\phi) |\lambda_{\phi}|=n$.
Furthermore the size of a conjugacy class with this data is \[
\frac{|GL(n,q)|}{\prod_{\phi} c_{\phi,\lambda_{\phi}}}
\] where $c_{\phi,\lambda_{\phi}}$ is an explicit function of
$\lambda_{\phi}$ and the degree of $\phi$.

    As in Theorem \ref{nearly}, define $D(\alpha)$ to be the
sum of the degrees (counted with multiplicity) of the irreducible factors
which occur with multiplicity greater than one in the characteristic
polynomial of $\alpha$. Theorem \ref{nearly} implies that if $a=c_1 +
\sqrt{\frac{c_2}{\epsilon}}$, then the chance that $D(\alpha) \leq a$ is
at least $1 - \epsilon$.

Thus it suffices to show that the proportion of elements $\alpha$ in
$GL(n,q)$ with $D(\alpha) = b \leq a$ and which fix a $k$-space goes to $0$
as $k \rightarrow \infty$. Let $c_g(z)$ denote the characteristic
polynomial of $g$. Note that the proportion of elements of
$GL(n,q)$ with $D(\alpha)=b$ and which fix a $k$-space is at most \[
\sum_{t=0}^b \sum_{\phi \in S_1(n-b)} \frac{1}{\prod_{\phi_i}
(q^{\deg(\phi_i)}-1)} \sum_{\psi \in S_2(\phi)} \frac{|g \in GL(b,q):
c_g(z) = \psi|}{|GL(b,q)|} .\] (Here $S_1(n-b)$ is the set of squarefree
monic polynomials of degree $n-b$ with nonzero constant term, with the
property that some subset of its factors have degrees adding to $k-t$. The
$\phi_i$ are the irreducible factors of $\phi$. The set $S_2(\phi)$ is the
set of monic polynomials $\psi$ of degree $b$ with nonzero constant term,
with the property that $\psi$ is relatively prime to $\phi$ and that all
irreducible factors occur with multiplicity greater than one). This formula
follows from the fact that any polynomial factors into a squarefree part
and a relatively prime part where all factors have multiplicity greater than
one.

    It is clear that \begin{eqnarray*} & & \sum_{t=0}^b \sum_{\phi
\in S_1(n-b)} \frac{1}{\prod_{\phi_i} (q^{\deg(\phi_i)}-1)} \sum_{\psi \in
S_2(\phi)} \frac{|g \in GL(b,q): c_g(z) = \psi|}{|GL(b,q)|}\\ & \leq &
\sum_{t=0}^b \sum_{\phi \in S_1(n-b)} \frac{1}{\prod_{\phi_i}
(q^{\deg(\phi_i)}-1)}
\end{eqnarray*} But \[ \sum_{\phi \in S_1(n-b)} \frac{1}{\prod_{\phi_i}
(q^{\deg(\phi_i)}-1)} \] is precisely the proportion of elements in
$GL(n-b,q)$ which are regular semisimple and which fix a $(k-t)$-space. By
Theorem \ref{Gltorus}, this is at most the proportion of elements in
$S_{n-b}$ which fix a $(k-t)$-set. Summing over $(t,b)$ with $0 \leq t
\leq b \leq a$, it follows from Theorem \ref{LPyber} that the proportion
of $\alpha \in GL(n,q)$ with $D(\alpha) \leq a$ and which fix a $k$-space
is at most $\frac{a^2 C}{(k-a)^{.01}}$ for a universal constant $C$.
This yields the theorem since $a=c_1 + \sqrt{\frac{c_2}{\epsilon}}$. \end{proof}

{\it Remark:} Taking $\epsilon=1/k^{.005}$ in Theorem \ref{GLlargek0} shows that
the probability of fixing a $k$-space is at most $A/k^{.005}$, for $A$ a universal
constant.

\subsection{SU}

    The results in this subsection parallel those in Subsection \ref{Glmainresults}.
Note that in analyzing the action of the unitary groups on nondegenerate
or totally singular $k$-spaces, one can suppose that $1 \leq k \leq n/2$.

    First we show that for fixed $q,k$, the large $n$ proportion of elements
which are regular semisimple and derangements in subspace actions is uniformly bounded away from $0$.

\begin{theorem} \label{eigenequalU} Let $gSU(n,q)$ be a coset of $SU(n,q)$ in
$U(n,q)$. For $k$ fixed, $q$ fixed, and $n \rightarrow \infty$, the
proportion of elements of $gSU(n,q)$ which are regular semisimple and derangements on
nondegenerate (resp. totally singular) $k$-spaces is equal to the proportion
of elements of $U(n,q)$ which are regular semisimple and derangements on nondegenerate (resp.
totally singular) $k$-spaces.
\end{theorem}

\begin{proof} The argument is the same as that of Theorem \ref{eigenequal}, where the notion
of bad conjugacy class is defined in the proof of Theorem \ref{differenceSU}.
\end{proof}

We next consider eigenvalue free elements in unitary groups.

\begin{lemma} \label{NpeigenfreeU}
\begin{enumerate}
\item The $n \rightarrow \infty$ proportion of elements of $U(n,q)$ which are
regular semisimple and eigenvalue free is equal to the $n \rightarrow \infty$ proportion
of regular semisimple elements of $U(n,q)$ divided by
\[ \left( 1+\frac{1}{q+1} \right)^{q+1} \left( 1+\frac{1}{q^2-1} \right)^{(q^2-q-2)/2} .\]

\item For $q = 2$ the proportion of part 1 is at least $.174$. For
$q \geq 3$ the proportion of part 1 is at least $.2$.

\item The proportion of elements of $U(n,q)$ which are regular semisimple and derangements on nondegenerate 1-spaces
is at least the proportion of part 1.

\item The proportion of elements of $U(n,q)$ which are regular semisimple and derangements on totally singular 1-spaces
is at least the proportion of part 1.

\end{enumerate}
\end{lemma}

\begin{proof} The paper \cite{FNP} uses generating functions to compute the large $n$
proportion of regular semisimple elements of $U(n,q)$. A very minor modifications of their
argument (removing terms corresponding to degree $1$ polynomials) yields part 1. Part 2
follows from part 1 and Theorem \ref{Uregss}.

Parts 3 and 4 follow since any eigenvalue free element of $U(n,q)$ is a
derangement on both nondegenerate and totally singular 1-spaces.
\end{proof}

    It is helpful to treat the cases $q=2,3$ separately.

\begin{theorem} \label{usmallq}
\begin{enumerate}
\item For $k \geq 2$ fixed, the $n \rightarrow \infty$ proportion of elements of
$U(n,2)$ which are regular semisimple and derangements on nondegenerate
$k$-spaces is at least $1/20$.

\item For $k \geq 2$ fixed, the $n \rightarrow \infty$ proportion of elements of
$U(n,3)$ which are regular semisimple and derangements on nondegenerate
$k$-spaces is at least $1/27$.
\end{enumerate}
\end{theorem}

\begin{proof} The stabilizer of a nondegenerate $k$-space is $U(k,q) \times U(n-k,q)$.
Hence by the logic of the $q=2$ case of Theorem \ref{correctSL}, the
proportion of elements in $U(n,q)$ which are regular semisimple and
derangements on nondegenerate $k$-spaces is at least the difference of the
proportion of regular semisimple elements in $U(n,q)$ and the proportion
of regular semisimple elements in $U(k,q) \times U(n-k,q)$. Since $k$ is
fixed and $n \rightarrow \infty$, by Theorems \ref{Uregss} and
\ref{Usmallregss}, the result follows for $q=2$ since $.414(1-.877)>.05$
and for $q=3$ since $.628(1-.94)>1/27$.
\end{proof}

\begin{theorem} Suppose that $1 \leq k \leq n/2$ is fixed. Then for all but
finitely many $(n,q)$ pairs, the proportion of elements in any coset
$gSU(n,q)$ in $U(n,q)$ which are regular semisimple and derangements on nondegenerate $k$-spaces
is at least $1/27$. \end{theorem}

\begin{proof} By the results of Section \ref{Large Fields}, the proportion of elements in the coset
$gSU(n,q)$ which are regular semisimple goes to 1 as $q \rightarrow
\infty$ uniformly in $n$. Using this with Theorems \ref{Utorus} and
\ref{Dix}, one concludes that
for any $\epsilon > 0$, the proportion of
elements in the coset $gSU(n,q)$ which are derangements on nondegenerate
$k$-spaces is at least $1/3-\epsilon$ for $q$ sufficiently large. This is
easily at least $1/27$.

    For $q$ fixed, Theorem \ref{eigenequalU} shows that it suffices to prove that the
proportion of elements of $U(n,q)$ which are regular semisimple and derangements on nondegenerate $k$-spaces
is at least $1/27$. By Theorem \ref{Uregss}, for $q \geq 4$ the large $n$ limiting proportion
    of regular semisimple elements of $U(n,q)$ is at least $.72$. Together
    with Theorems \ref{Utorus} and \ref{Dix}, this implies that the $n \rightarrow
    \infty$ proportion of elements of $U(n,q)$ which are regular semisimple and
     derangements on nondegenerate $k$-spaces is at least
    $.72-2/3 > 1/27$. For $q=2,3$, the result follows from Theorem
    \ref{usmallq} and Lemma \ref{NpeigenfreeU}.
\end{proof}

\begin{theorem} Suppose that $1 \leq k \leq n/2$ is fixed. Then for all but
finitely many $(n,q)$ pairs, the proportion of elements in any coset
$gSU(n,q)$ of $U(n,q)$ which are regular semisimple and derangements on
totally singular $k$-spaces is at least $1/26$. \end{theorem}

\begin{proof} By the results of Section \ref{Large Fields}, it
suffices to take $q$ fixed. For $q$ fixed, Theorem \ref{eigenequalU} shows
that it suffices to prove that the proportion of elements of $U(n,q)$ which
are regular semisimple and derangements on totally singular $k$-spaces is at least 1/26. Theorem \ref{Uregss}
gives that the $n \rightarrow \infty$ proportion of regular semisimple elements in
$U(n,q)$ is at least $.628$ for $q \geq 3$. By Theorem \ref{Utorus} and
Theorem \ref{evenkset}, the proportion of elements in $U(n,q)$ which are
regular semisimple and fix a totally singular $k$-space is at most $1/2$.
This proves the theorem for $q > 2$ since $.628-1/2 = .128 > 1/26$.

The final case to consider is $q=2$. From Theorem \ref{Uregss}, the large $n$ limiting
proportion of regular semisimple elements in $U(n,2)$ is at least $.414$.
By Theorem \ref{Utorus} and Theorem \ref{evenkset}, the chance that an
element of $U(n,2)$ is regular semisimple and fixes a totally singular
$k$-space is at most $3/8<.414$ for $k \geq 2$. The result follows in this
case since $.414 - 3/8 \geq 1/26$. Thus the only remaining case is
$k=1,q=2$, and this follows from Lemma \ref{NpeigenfreeU}.
\end{proof}

Next we treat the case that $k \rightarrow \infty$.   Recall
that for $q \rightarrow \infty$, we already have very good estimates
on the proportion of derangements (see Theorem
\ref{thm:   large q}).

\begin{theorem} \label{Ulargek0} Suppose that $1 \leq k \leq n/2$.
\begin{enumerate}

\item For $q$ fixed, and $k \rightarrow \infty$, the proportion of elements
of $U(n,q)$ which are derangements on nondegenerate $k$-spaces $\rightarrow 1$. More precisely,
there are universal constants $A,B$ such that for any $\epsilon > 0$, and $k$,
the proportion of elements of $U(n,q)$ which are
derangements is at least \[ 1 - \epsilon - \frac{A}{\epsilon (k - B/\sqrt{\epsilon})^{.01}}. \]

\item For $q$ fixed, and $k \rightarrow \infty$, the proportion of elements
of $U(n,q)$ which are derangements on totally
singular $k$-spaces $\rightarrow 1$. More precisely,
there are universal constants $A,B$ such that for any $\epsilon > 0$, and $k$,
the proportion of elements of $U(n,q)$ which are derangements is at least
\[ 1 - \epsilon - \frac{A}{\epsilon (k - B/\sqrt{\epsilon})^{.5}}. \]

\end{enumerate}
\end{theorem}

\begin{proof} For part 1 we argue as follows. As in the proof of Theorem \ref{nearly}, define
$D(\alpha)$ to be the sum of the degrees (counted with multiplicity) of
the irreducible factors which occur with multiplicity greater than one in
the characteristic polynomial of $\alpha$. Theorem \ref{nearly} implies
that if $a=c_1+ \sqrt{\frac{c_2}{\epsilon}}$, then the chance that
$D(\alpha) \leq a$ is at least $1-\epsilon$. Thus it suffices to show that
the proportion of elements $\alpha$ in $U(n,q)$ with $D(\alpha) = b \leq
a$ and which fix a nondegenerate $k$-space goes to $0$ as $k \rightarrow
\infty$.

    So we study the proportion of elements of $U(n,q)$ with
    $D(\alpha)=b$ and which fix a nondegenerate $k$-space. For any vector
    space $V$, an element $g$ of $U(V)$ has its characteristic polynomial
    expressible as $f(z) \cdot h(z)$, where $f$ is multiplicity free, $h$ is prime
    to $f$, and $f$ is closed under the $q$-Frobenius. Then $V$ is the
    direct sum of the kernels of $f(g)$ and $h(g)$. Applying this to any
    $g$-invariant subspace, it follows that the proportion of elements of $U(n,q)$
    with $D(\alpha)=b$ and which fix a nondegenerate $k$-space is at most the sum
    as $t$ goes from $0$ to $b$ of the proportion of elements of $U(n-b,q)$
    which are regular semisimple and fix a nondegenerate $k-t$ space.
    Arguing as in the general linear case (Theorem \ref{GLlargek0}), the
    result now follows from Theorems \ref{Utorus} and \ref{LPyber}.

    The proof of part two is nearly identical, except that one uses Theorems
    \ref{Utorus} and \ref{evenkset}. \end{proof}

{\it Remark:} Taking $\epsilon=1/k^{.005}$ in part 1 of Theorem \ref{Ulargek0} shows
that the chance of fixing a nondegenerate $k$-space is at most $A/k^{.005}$ for a universal
constant $A$. Taking $\epsilon=1/k^{.25}$ in part 2 of Theorem \ref{Ulargek0} shows
that the chance of fixing a totally singular $k$-space is at most $A/k^{.25}$ for a universal
constant $A$.

\subsection{Sp}

    This section considers the symplectic groups. For the action on nondegenerate
$2k$-spaces, we suppose that $1 \leq k \leq n/2$.  Of course a totally singular space
has dimension at most $n$. In even characteristic one must also consider the action on
nondegenerate hyperplanes (viewing $Sp(2n,q)$ as $\Omega(2n+1,q)$).

    To begin we discuss the case of $k$ fixed. First we treat nondegenerate subspaces.

\begin{theorem} \label{manycases}
Let $1 \leq k \leq n/2$ be fixed. The $n \rightarrow \infty$ proportion of
elements in $Sp(2n,q)$ which are regular semisimple and derangements on
nondegenerate $2k$-spaces is at least $.11$ for $q=2$, $.05$ for $q=3$,
$.11$ for $q=4$, $.13$ for $q=5$, $.1$ for $q=7$, and $.08$ for $q=8$.
\end{theorem}

\begin{proof} The stabilizer of a nondegenerate $2k$-space in $Sp(2n,q)$ is
 $Sp(2k,q) \times Sp(2n-2k,q)$. By Proposition \ref{7/12bound} and part 3 of Theorem
 \ref{Sptorus1}, it follows that the proportion of regular semisimple
 elements in $Sp(2n,2)$ or $Sp(2n,3)$ is at most $7/12$ and $5/6$ respectively.
 Hence the reasoning of Theorem \ref{correctSL} for $q=2$, together with Theorem
 \ref{Spregss} implies that for sufficiently large $n$ the proportion
 of elements in $Sp(2n,q)$ which are regular semisimple and
 derangements on nondegenerate $2k$-spaces is at least $.283[1-7/12]
 \geq .11$. Similarly one sees that for $q=3$ the proportion of
 derangements on nondegenerate $2k$-spaces is at least $.348[1-5/6]
 \geq .05$. Recall that Theorem \ref{boundedsmallSp} gives that the
 proportion of regular semisimple elements in $Sp(2n,q)$ is at most
 .74 for $q=4$, .80 for $q=5$, .86 for $q=7$, and .88 for $q=8$. Using the
 same reasoning as for $q=2,3$ one concludes that the $n \rightarrow
 \infty$ proportion of elements of $Sp(2n,q)$ which are regular
 semisimple and derangements on nondegenerate $2k$-spaces is at least
 $.453[1-.74] \geq .11$ for $q=4$, at least $.654[1-.80] \geq .13$
 for $q=5$, at least $.745[1-.86] \geq .1$ for $q=7$, and at least
 $.686[1-.88] \geq .08$ for $q=8$.
\end{proof}

\begin{theorem} Suppose that $1 \leq k < n/2$ is fixed.
Then for all but finitely many $(n,q)$ pairs, the proportion of elements
in $Sp(2n,q)$ which are  regular semsimple and  derangements on nondegenerate 2$k$-spaces is at
least $1/20$. \end{theorem}

\begin{proof} By Theorem \ref{uniformSp}, when $q \rightarrow \infty$
the proportion of regular semisimple elements in $Sp(2n,q)$ goes to 1
uniformly in $n$. Hence for $q$ sufficiently large it follows from
Theorems \ref{Sptorus1} and \ref{Dix} that the proportion of derangements
on nondegenerate 2$k$-spaces is at least $1/3-\epsilon$ for any
$\epsilon>0$. This is easily more that $1/20$.

    Suppose that $q \geq 9$ is fixed. By Theorem \ref{Spregss},
    the $n \rightarrow \infty$ limiting proportion of regular
    semisimple elements in $Sp(2n,q)$ is at least $.797$. By
    Theorems \ref{Sptorus1} and \ref{Dix}, the proportion of
    elements which are regular semisimple and fix a nondegenerate $2k$-space
    is at most $2/3$. The theorem is proved for $q \geq 9$ since
    $.797-2/3 \geq .13$. The remaining cases follow from Theorem
    \ref{manycases}. \end{proof}

 Note that if $g \in Sp(2n,q)$ is semisimple and fixes a totally singular $k$-space $W$ with $k < 2n$, then
 $g$ also fixes a nondegenerate $2k$-space (it fixes a complement $U$ to $W$ in $W^{\perp}$; since
 $U$ is nondegenerate,  the $x$-invariant space is $U^{\perp}$).   Thus,

 \begin{cor} \label{cor:tssymp} Suppose that $1 \leq k < n/2$ is fixed.
Then for all but finitely many $(n,q)$ pairs, the proportion of elements
in $Sp(2n,q)$ which are  regular semsimple and  derangements on totally singular $k$-spaces is at
least $1/20$.
\end{cor}

To complete the discussion of actions on nondegenerate spaces,
recall that in characteristic 2 there is the action of $Sp(2n,q)$ on
nondegenerate hyperplanes (in the indecomposable orthogonal representation of
dimension $2n+1$).    We recall that semisimple  elements of $\Omega^{\pm}(2n,q)$ are
strongly regular if they do not have $1$ as an eigenvalue (or equivalently (since $q$ is even)  are regular
semisimple in $Sp(2n,q)$).

\begin{lemma} \label{lem:stronglyregular3}
Let $G=Sp(2n,q)$ with $q$ even and fixed.  Let $R^{\epsilon}$
denote the set of regular semisimple elements of $G$ contained in
some conjugate of $\Omega^{\epsilon}(2n,q)$.
\begin{enumerate}[(1)]
\item $R^+ \cap R^-$ is the empty set.
\item  $R^+ \cup R^-$ is the set of regular semisimple elements in $G$.
\item  $\lim_{n \rightarrow \infty} | |R^+|- |R^-| |/|G| =0$.
\end{enumerate}
\end{lemma}

\begin{proof}  Note that a regular semisimple element $x \in G$ fixes precisely
one nondegenerate space in the orthogonal module $V$  (namely
$[x,V]=(x-1)V$).   Thus, the first two statements follow immediately.

Since each semisimple regular element of $G$ lives in precisely one
orthogonal group, it follows that the ratio
$|R^{\epsilon}|/|G|$ is precisely $(1/2)$ the ratio of strongly regular
semisimple elements in $\Omega^{\epsilon}(2n,q)$.

If $x$ is a regular semisimple element in $\Omega^{\pm}(2n,q)$
and has eigenvalue $1$, then $x$ is trivial on a nondegenerate
$2$-space and strongly semisimple regular on the orthogonal complement
of that $2$-space.   Using the fact that the limiting proportion of regular
semsimple elements in $\Omega^{\pm}(2n,q)$ does not depend on the
type (Corollary \ref{regssO}) and counting in terms of nondegenerate subspaces of codimension $2$,
(3) follows easily.
\end{proof}

\begin{theorem} Let $q$ be even.
\begin{enumerate}
\item For $|Sp(2n,q)|$ sufficiently large, the proportion of elements which are
regular semisimple and
derangements on nondegenerate positive type hyperplanes is at least $.14$.

\item For $|Sp(2n,q)|$ sufficiently large, the proportion of
 elements which are regular semisimple and
derangements on nondegenerate negative type hyperplanes is at least $.14$.
\end{enumerate}
 \end{theorem}

\begin{proof}
By Theorem \ref{uniformSp}, for $q$ sufficiently large the proportion
of regular semisimple elements in $Sp(2n,q)$ goes to $1$ uniformly in $n$.
Then the theorem follows  by Theorem \ref{Sptorus2} (and the proportion
of regular semsimple elements which are derangements approaches $1/2$).

By the previous result, we see that the limiting proportion of regular semisimple
elements of $Sp(2n,q)$ which a fix a nondegenerate hyperplane of given type
is precisely $(1/2)$ the proportion of regular semisimple elements.

 By Theorem
    \ref{Spregss},  the $n \rightarrow \infty$
    limiting proportion of regular semisimple elements of
    $Sp(2n,q)$ is at least $.283$, whence the result.
 \end{proof}

Next we note that in  the case of totally singular $1$-spaces,
using results of   Neumann and Praeger, we get a better bound than Corollary \ref{cor:tssymp}.

\begin{lemma} \label{NpeigenfreeSp} \rm{(\cite{NP})}
\begin{enumerate}
\item Suppose that $q$ is even. Then the $n \rightarrow \infty$ proportion
of elements in $Sp(2n,q)$ which fix no totally singular one-dimensional subspaces is \[
\prod_{i \geq 1} \left( 1-\frac{1}{q^{2i-1}} \right) \prod_{i \geq 1}
\left( 1-\frac{1}{q^i} \right)^{(q-2)/2} \geq .4.\]

\item Suppose that $q$ is odd. Then the $n \rightarrow \infty$ proportion
of elements in $Sp(2n,q)$ which fix no totally singular one-dimensional subspaces is \[
\prod_{i \geq 1} \left( 1-\frac{1}{q^{2i-1}} \right)^2 \prod_{i \geq 1}
\left( 1-\frac{1}{q^i} \right)^{(q-3)/2} \geq .4.\]
\end{enumerate}
\end{lemma}

Next we consider the case $k \rightarrow \infty$. Recall that for $q \rightarrow
\infty$, we already have very good estimates on the proportion of
derangements (see Theorem \ref{thm: large q}).

\begin{theorem} \label{ShalargekSp}
\begin{enumerate}

\item Suppose that $1 \leq k \leq n/2$. For $q$ fixed and $k \rightarrow \infty$, the proportion of elements of
$Sp(2n,q)$ which are derangements on nondegenerate 2$k$-spaces converges to
1. More precisely, there are universal constants $A,B$ such that for any $\epsilon > 0$ and $k$,
the proportion of derangements is
at least \[ 1 -\epsilon - \frac{A}{\epsilon (k - B/\sqrt{\epsilon})^{.01}}. \]

\item Suppose that $1 \leq k \leq n$. For $q$ fixed and $k \rightarrow \infty$, the proportion of elements of
$Sp(2n,q)$ which are derangements on totally singular $k$-spaces converges
to 1. More precisely, there are universal constants $A,B$ such that for any $\epsilon > 0$ and $k$,
the proportion of derangements is
at least \[ 1 -\epsilon - \frac{A}{\epsilon (k - B/\sqrt{\epsilon})^{.5}}. \]

\end{enumerate}
\end{theorem}

\begin{proof} Given an element $g$ of $Sp(2n,q)$, one can split it into a regular
semisimple part and non-regular semisimple part (i.e. a part with a
square-free characteristic polynomial $f(z)$ and a relatively prime
polynomial $h(z)$ where all factors have multiplicity greater than 1). If
$g$ fixes a nondegenerate $k$-space, then for some $t$, the regular
semisimple part fixes a nondegenerate $k-t$ space, and the non-regular
semisimple part fixes a nondegenerate $t$ space. This is true since if
$W$ is any nondegenerate invariant space for $g$, then $W$ is the sum of
$W \cap ker(f(g))$ and $W \cap ker(h(g))$. Then arguing as in the general
linear and unitary cases (Theorems \ref{GLlargek0} and \ref{Ulargek0}),
using Theorems \ref{Sptorus1} and \ref{LPyber}, one proves part 1.

For part 2, one can replace nondegenerate by totally singular in the
previous paragraph, and then use Theorems \ref{Sptorus1} and
\ref{reducetounitary}. \end{proof}

{\it Remark:} Taking $\epsilon=1/k^{.005}$ in part 1 of Theorem \ref{ShalargekSp} shows
that the chance of fixing a nondegenerate $k$-space is at most $A/k^{.005}$ for a universal
constant $A$. Taking $\epsilon=1/k^{.25}$ in part 2 of Theorem \ref{ShalargekSp} shows
that the chance of fixing a totally singular $k$-space is at most $A/k^{.25}$ for a universal
constant $A$.

\subsection{O}

    This section studies the proportion of derangements in subspace actions
of $\Omega$. Note that when $q$ is even, the case $\Omega(2n+1,q)$ can be
disregarded given that it is isomorphic with $Sp(2n,q)$.

    First we treat the case of fixed $k$ and even $q$, starting with $k=1$.

\begin{lemma} \label{Omegaeigenfree} Let $q$ be even. Then the $n \rightarrow \infty$
proportion of eigenvalue free elements of $\Omega^{\pm}(2n,q)$ is
\[ \prod_{i \geq 1} \left( 1-\frac{1}{q^{2i-1}} \right) \prod_{i \geq 1}
\left( 1-\frac{1}{q^i} \right)^{(q-2)/2} \geq .4.\]
\end{lemma}

\begin{proof} It is easy to see that all eigenvalue free elements of $O^{\pm}(2n,q)$
are in $\Omega^{\pm}(2n,q)$. The $n \rightarrow \infty$ proportion of
eigenvalue free elements in $O^{\pm}(2n,q)$ was calculated in \cite{NP}
(using generating functions) to be \[ \frac{1}{2} \prod_{i \geq 1} \left(
1-\frac{1}{q^{2i-1}} \right) \prod_{i \geq 1} \left( 1-\frac{1}{q^i}
\right)^{(q-2)/2}.\] The inequality is from Lemma \ref{NpeigenfreeSp}.
\end{proof}

    Lemma \ref{Omegaeigenfree} immediately handles the case of the action of
$\Omega^{\pm}(2n,q)$ on 1-spaces where the quadratic form doesn't vanish,
in even characteristic.

\begin{cor} Let $q$ be even. Then the $n \rightarrow \infty$ proportion of
elements of $\Omega^{\pm}(2n,q)$ which are derangements on the set of
lines where the quadratic form does not vanish is at least $.4$.
\end{cor}

Next we treat more general nondegenerate spaces.  Note that for $q$ even,
any odd dimensional subspace has a radical (with respect to the corresponding
alternating form) and so the only time the stabilizer of an odd dimensional
space is maximal is when it has dimension $1$.    We next consider
strongly regular semisimple elements.  When $q$ is even, strongly
regular semisimple elements are precisely regular semisimple elements that do
not have $1$ as eigenvalue.  In particular, any strongly regular semisimple element
is a derangement on nondegenerate $1$-dimensional spaces.

\begin{lemma} \label{lem:stronglyregeven}  Let $q$ be even.  Let $G:=\Omega^{\pm}(2n,q)$.
For all but finitely many $n$ and $q$,   the proportion of strongly regular semisimple elements in
$G$ is greater than $.28$.
\end{lemma}

\begin{proof}    If $q \rightarrow \infty$, we have seen that the proportion of regular semisimple
elements in $Sp(2n,q) \rightarrow 1$, whence by Lemma \ref{lem:stronglyregular3},
the same is true for $G$.

Now fix $q$.
By the proof of Lemma \ref{lem:stronglyregular3}, the limiting proportion of
strongly regular semisimple elements in $G$ is the same as the limiting proportion of regular
semisimple elements for $Sp(2n,q)$.    By Theorem
 \ref{Spregss}, this limit is greater than $.28$.
\end{proof}

 Since strongly regular semisimple elements are precisely those regular semisimple elements
 which are derangements on nondegenerate $1$-spaces, we see:

\begin{cor}  \label{cor:nondeg1qeven}   Let $q$ be even.
 For all but finitely many $(n,q)$, the proportion of elements which are  regular semisimple and derangements on
nondegenerate $1$-spaces in $\Omega^{\pm}(2n,q)$ is greater than $.28$.
\end{cor}

It remains  to deal with even
dimensional nondegenerate spaces.

\begin{theorem}  \label{thm:spevennon}  Suppose that $1 \leq k < n$ is fixed. Let $q$ be even.
For all but finitely many pairs $(n,q)$, the proportion of elements in
$\Omega^{\pm}(2n,q)$ which are regular semisimple and derangements on
nondegenerate 2$k$-spaces of positive (resp. negative) type is at least
$.056$.
\end{theorem}

\begin{proof} We prove the result for the case of positive type spaces since the
negative case can be handled by replacing the word positive by the word
negative in all places.

    If $q \rightarrow \infty$, by Theorem \ref{GL result}, the proportion of regular
semisimple elements in $\Omega^{\pm}(2n,q)$ goes to $1$ for $q$
sufficiently large uniformly in $n$. The result then follows from Theorems
\ref{Otorus2} and \ref{1stDn} since $1-1/2>.056$.

    Suppose that $q$ is fixed. For $q \geq 4$, by Corollary \ref{regssO}
the $n \rightarrow \infty$ limiting proportion of regular semisimple
elements in $\Omega^{\pm}(2n,q)$ is $(1+\frac{q}{q^2-1})$ multiplied by
the corresponding limit for the symplectic groups (given by Theorem
\ref{Spregss}) and hence is at least .573. The result follows from Theorems
\ref{Otorus2} and \ref{1stDn} since $.573-1/2 \geq .056$. For $q=2$, Corollary
\ref{regssO} and Theorem \ref{Spregss} imply that the the $n \rightarrow
\infty$ limiting proportion of regular semisimple elements in
$\Omega^{\pm}(2n,2)$ is at least .47. The result follows from Theorems
\ref{Otorus2} and \ref{2ndDn} since $.47-.414 \geq .056$. \end{proof}

    Next we analyze the case of totally singular $k$-spaces when $q$ is even.

\begin{theorem} Suppose that $1 \leq k < n$ is fixed. Let $q$ be even.
For all but finitely many pairs $(n,q)$, the proportion of elements in
$\Omega^{\pm}(2n,q)$ which are regular semisimple and derangements on totally
singular $k$-spaces is at least $.056$.
\end{theorem}

\begin{proof}   Note that if $x$ is regular semisimple and fixes a totally
singular $k$-dimensional space, then $x$ fixes a nondegenerate $2k$-space
of $+$ type, whence the result follows by Theorem \ref{thm:spevennon}.
\end{proof}

Next we consider the case of $q$ odd. We begin with the case of 1-spaces;
some related results are in \cite{NP}.

\begin{theorem} \label{eigenfreeoddO} Let $q$ be odd and fixed.
\begin{enumerate}
\item The $n \rightarrow \infty$ limiting proportion of regular semisimple
eigenvalue free elements in $\Omega^{\pm}(2n,q)$ is equal to the
corresponding limiting proportion for $SO^{\pm}(2n,q)$. This proportion is
equal to $(1+\frac{1}{q-1})^{-(q-3)/2}$ multiplied by the limiting
proportion of regular semisimple elements in the symplectic groups (given
in Theorem \ref{Spregss}). For $q \geq 3$ this product is at least $.348$.

\item The $n \rightarrow \infty$ limiting proportion of
elements in $\Omega(2n+1,q)$ which are regular semisimple and
derangements on positive (resp.
negative) type 1-spaces is equal to the limiting proportion for
$SO(2n+1,q)$. For $q \geq 3$ this proportion is at least $\frac{1}{2}
(1+\frac{1}{q-1})^{-(q-3)/2}$ multiplied by the limiting proportion of
regular semisimple elements in the symplectic groups (given in Theorem
\ref{Spregss}); hence this proportion is at least $.174$.
\end{enumerate}
\end{theorem}

\begin{proof} For part 1 of the theorem, the argument of Theorem
\ref{Omegaregssequal} implies that the $n \rightarrow \infty$ limiting
proportion of eigenvalue free regular semisimple elements in
$\Omega^{\pm}(2n,q)$ is equal to the corresponding proportion in
$SO^{\pm}(2n,q)$. Indeed, to bound the difference between the proportions,
instead of summing over all bad Weyl group conjugacy classes, one sums
only over bad Weyl group conjugacy classes without fixed points. Letting
$t^{\pm}_n$ denote the number of regular semisimple eigenvalue free
elements of $SO^{\pm}(2n,q)$, using the methods of \cite{FNP} one obtains
that
\begin{eqnarray*}
& & 1+ \sum_{n \geq 1} u^n \left( \frac{t^+_n}{|O^+(2n,q)|} + \frac{t^-_n}{|O^-(2n,q)|} \right) \\
& = &  \prod_{d \geq 1} (1+\frac{u^d}{q^d+1})^{N^*(q;2d)} \prod_{d \geq 2}
(1+\frac{u^d}{q^d-1})^{M^*(q;d)}
\end{eqnarray*} which one recognizes as $(1+\frac{u}{q-1})^{-(q-3)/2}$ multiplied
by the generating function for regular semisimple elements in the
symplectic groups. From this and an analysis of the difference of the the
generating functions for $t_n^+$ and $t_n^-$ (showing its contribution to
be negligible), part 1 of the theorem follows.

    For part 2 of the theorem (as in part 1), the equality of the large $n$
limiting proportions follows by the technique of Theorem
\ref{Omegaregssequal}. Next, observe that a regular semisimple element
$\alpha$ in $SO(2n+1,q)$ is a derangement on positive (resp. negative)
1-spaces if it is eigenvalue free except for the $z-1$ factor, which has
negative (resp. positive) type and occurs with multiplicity one. The
result now follows from a generating function argument similar to that in
the previous paragraph. \end{proof}

    Next we consider the proportion of derangements in the action of
$\Omega^{\pm}(n,q)$ on nondegenerate and totally singular subspaces.
Theorem \ref{regssOK} shows that if one restricts to regular semisimple
elements, then as $n \rightarrow \infty$ it is sufficient to work in
$SO^{\pm}(n,q)$.

\begin{theorem} \label{regssOK} Let $q$ be odd and fixed.
\begin{enumerate}
\item The $n \rightarrow \infty$ limiting proportion of regular semisimple derangements
in $\Omega^{\pm}(n,q)$ on nondegenerate $k$-spaces of positive (resp.
negative) type is equal to the corresponding limit for $SO^{\pm}(n,q)$.

\item The $n \rightarrow \infty$ limiting proportion of regular semisimple derangements
in $\Omega^{\pm}(n,q)$ on totally singular $k$-spaces is equal to the
corresponding limit for $SO^{\pm}(n,q)$.
\end{enumerate}
\end{theorem}

\begin{proof} Both parts follow by the technique of Theorem \ref{Omegaregssequal},
since instead of summing over all bad conjugacy classes in the Weyl group,
one only sums over those bad conjugacy classes which could correspond
(this correspondence was discussed in Section \ref{maximaltori}) to
regular semisimple derangements. \end{proof}

\begin{theorem} \label{Omeganondegenodd} Let $q$ odd and $1 \leq k \leq n$ be
fixed. For all but finitely many $(n,q)$ pairs, the proportion of elements
in $\Omega(2n+1,q)$ which are semisimple regular and are derangements on nondegenerate $k$-spaces of
positive (resp. negative) type is at least $.07.$
\end{theorem}

\begin{proof} By Theorem \ref{GL result}, the proportion of elements in
$\Omega(2n+1,q)$ which are regular semisimple goes to 1 as $q
\rightarrow \infty$ uniformly in $n$. Hence for $q$ sufficiently large, it
follows from Theorems \ref{Otorus1} and \ref{Dix} that
the proportion of derangements on nondegenerate $k$-spaces is at least
$1/3$ which is bigger than $.07$.

For $q \geq 5$, Theorems \ref{addone} and \ref{Omegaregssequal}
give that the $n \rightarrow \infty$ proportion of strongly
regular semisimple elements in $\Omega(2n+1,q)$ is at least $.654$. From
Theorem \ref{Otorus1} and part 3 of Theorem \ref{1stDn}, the proportion of elements which are
strongly regular semisimple elements in $\Omega(2n+1,q)$ and fix a nondegenerate
$k$-space is at most $1/2$. The result follows for $q \geq 5$ since
$.654-1/2 \geq .07$. If $q=3$, Theorems \ref{addone} and \ref{Omegaregssequal} give that the
$n \rightarrow \infty$ proportion of strongly regular semisimple elements in $\Omega(2n+1,q)$ is at least
$.348$. The result now follows from Theorems \ref{Otorus1} and
\ref{newthe3} since $.348-.276 \geq .07$. \end{proof}

\begin{theorem} \label{Omeganondegeneven} Let $q$ be odd and $1 \leq k
\leq n$ be fixed. For all but finitely many $(n,q)$ pairs, the proportion
of elements in $\Omega^{\pm}(2n,q)$ which are  regular semisimple and derangements on
nondegenerate $k$-spaces of positive (resp. negative) type is at least
$.07$.
\end{theorem}

\begin{proof} By Theorem \ref{GL result}, the proportion of elements in
$\Omega^{\pm}(2n,q)$ which are regular semisimple goes to 1 as $q
\rightarrow \infty$ uniformly in $n$. Hence for $q$ sufficiently large, it
follows from Theorems \ref{Otorus2}, \ref{Otorus3} and \ref{Dix} that
the proportion of derangements on nondegenerate $k$-spaces is at least
$1/3$ which is bigger than $.07$.

    Thus we can suppose that $q$ is fixed.   First assume that
$k$ is even. By Theorems \ref{addtwo} and \ref{Omegaregssequal}, for $q \geq 5$,
the $n \rightarrow \infty$ proportion of strongly regular semisimple elements
in $\Omega^{\pm}(2n,q)$ is at least $.654$. Thus Theorems \ref{Otorus2},
\ref{Otorus3} and \ref{1stDn} imply the result since $.654-1/2 \geq .07$.
If $q=3$, Theorems \ref{addtwo} and \ref{Omegaregssequal} give that the
$n \rightarrow \infty$ proportion of strongly regular semisimple elements in
$\Omega^{\pm}(2n,q)$ is at least $.348$. The result follows from Theorems
\ref{Otorus2}, \ref{Otorus3} and \ref{newthe3} since $.348-.276 \geq .07$.

Next suppose that $k$ is odd.  Then any semisimple element fixing
a nondegenerate $k$-dimensional space has a two dimensional
eigenspace corresponding to an eigenvalue $\pm 1$.  Now apply
Theorem \ref{eigenfreeoddO}.
\end{proof}

\begin{theorem} \label{Omegatotsingnodd} Let $q$ be odd and $1 \leq k < n$ be fixed.
For all but finitely many $(n,q)$ pairs, the proportion of elements in
$\Omega(2n+1,q)$ which are regular semisimple and  derangements on totally singular $k$-spaces is at
least $.07.$
\end{theorem}

\begin{proof}   Note that
if a semisimple element fixes a totally singular $k$-space, then it fixes
a nondegenerate $2k$-space (of $+$ type).   Now apply Theorem \ref{Omeganondegenodd},
noting that its proof actually gave a lower bound for the proportion of regular semisimple
derangements.
\end{proof}

\begin{theorem} \label{Omegatotsingneven} Let $q$ be odd and $1 \leq k \leq n$
be fixed. For all but finitely many $(n,q)$ pairs, the proportion of
elements in $\Omega^{\pm}(2n,q)$ which are regular semisimple and derangements on totally
singular $k$-spaces is at least $.15$. \end{theorem}

\begin{proof} By Theorem \ref{GL result}, the proportion of elements in
$\Omega^{\pm}(2n,q)$ which are regular semisimple goes to 1 as $q
\rightarrow \infty$ uniformly in $n$. Hence for $q$ sufficiently large and
$k<n$, it follows from Theorems \ref{Otorus2}, \ref{Otorus3} and
\ref{reducetounitary2} that the proportion of elements which are regular
semisimple and  derangements on totally
singular $k$-spaces is at least $1/2$ which is bigger than $.15$.
For $q$ sufficiently large and $1<k=n$, it follows from Theorems
\ref{Otorus2}, \ref{Otorus3} and part 2 of Theorem \ref{reducetounitary2}
that the proportion of elements which are regular semisimple and derangements on totally singular $n$-spaces is at
least $1-2(3/8)>.15$. For $k=1$, the result follows by Theorem
\ref{eigenfreeoddO}.

    Thus we can suppose that $q$ is fixed. Theorems \ref{largeregsseveno}
and \ref{Omegaregssequal} give that for $q \geq 3$, the $n \rightarrow \infty$ proportion of
regular semisimple elements in $\Omega^{\pm}(2n,q)$ is at least $.657$. From
Theorems \ref{Otorus2}, \ref{Otorus3} and \ref{reducetounitary2}, the
proportion of elements which are regular semisimple   in $\Omega^{\pm}(2n,q)$ and fix a
totally singular $k$-space is at most $1/2$. The result follows since
$.657 -1/2 > .15$.
\end{proof}

    To conclude, we treat the case $k \rightarrow \infty$. Recall that
for $q \rightarrow \infty$, we already have very good estimates on the proportion
of derangements (see Theorem \ref{thm: large q}).

\begin{theorem} \label{last} Suppose that $1 \leq k \leq n/2$.
\begin{enumerate}

\item For $q$ fixed and $k \rightarrow \infty$, the proportion of elements in
$\Omega^{\pm}(n,q)$ which are derangements on nondegenerate $k$-spaces converges to 1.
More precisely, there are universal constants $A,B$ such that for any $\epsilon >0$
and $k$, the proportion of derangements is
at least \[ 1 -\epsilon - \frac{A}{\epsilon (k - B/\sqrt{\epsilon})^{.01}}. \]

\item For $q$ fixed and $k \rightarrow \infty$, the proportion of elements in
$\Omega^{\pm}(n,q)$ which are derangements on totally singular $k$-spaces
converges to 1. More precisely, there are universal constants $A,B$ such that for any $\epsilon >0$
and $k$, the proportion of derangements is
at least \[ 1 -\epsilon - \frac{A}{\epsilon (k - B/\sqrt{\epsilon})^{.5}}. \]

\end{enumerate}
\end{theorem}

\begin{proof} For $q$ even, the proof is nearly identical to the symplectic case
(Theorem \ref{ShalargekSp}), and we omit further details. For $q$ odd, we work in
$SO$ instead of in $\Omega$, so that generating functions can be used. Clearly the
result for $SO$ implies the results for $\Omega$ (with different universal constants).
\end{proof}

{\it Remark:} Taking $\epsilon=1/k^{.005}$ in part 1 of Theorem \ref{last} shows
that the chance of fixing a nondegenerate $k$-space is at most $A/k^{.005}$ for a universal
constant $A$. Taking $\epsilon=1/k^{.25}$ in part 2 of Theorem \ref{last} shows
that the chance of fixing a totally singular $k$-space is at most $A/k^{.25}$ for a universal
constant $A$.

\section{Acknowledgements} Fulman was partially supported by NSA grants
H98230-13-1-0219 and and Simons Foundation Fellowship 229181. Guralnick was partially supported by
NSF grants DMS-1001962 and DMS-1302886 and by Simons Foundation Fellowship 224965.
The authors thank the referee for helpful comments.

\end{document}